\documentclass[graybox]{svmult}

\usepackage{mathptmx}       % selects Times Roman as basic font
\usepackage{helvet}         % selects Helvetica as sans-serif font
\usepackage{courier}        % selects Courier as typewriter font
\usepackage{type1cm}        % activate if the above 3 fonts are
                            \usepackage[mathscr]{eucal}
\usepackage{amssymb}
\usepackage{amsmath}
\usepackage{graphicx}
\usepackage{float}
\usepackage{epstopdf}
\usepackage{url}
\usepackage{enumerate}

\newcommand\be{\begin{enumerate}}
\newcommand\ee{\end{enumerate}}

\usepackage{makeidx}         % allows index generation
\usepackage{graphicx}        % standard LaTeX graphics tool
\usepackage{multicol}        % used for the two-column index
\usepackage[bottom]{footmisc}% places footnotes at page bottom

\makeindex             % used for the subject index
\begin{document}

\title*{On cubic difference equations with variable coefficients 
 and fading stochastic perturbations}
\titlerunning{Cubic difference equation with variable coefficients}
% Use \titlerunning{Cubic difference equation with variable coefficients} for an abbreviated version of
\author{Ricardo Baccas, C\'onall Kelly, and Alexandra Rodkina}
\authorrunning{R. Baccas, C. Kelly, A. Rodkina}
% Use \authorrunning{R. Baccas, C. Kelly, A. Rodkina} for an abbreviated version of
\institute{Ricardo Baccas \at Department of Mathematics, University of the West Indies, Mona, Kingston 7, Jamaica.\\ \email{ricardo.baccas02@uwimona.edu.jm}
\and C\'onall Kelly \at School of Mathematical Sciences, University College Cork, Ireland. \\\email{conall.kelly@ucc.ie}
\and Alexandra Rodkina \at Department of Mathematics, University of the West Indies, Mona, Kingston 7, Jamaica.\\ \email{alexandra.rodkina@uwimona.edu.jm}}
\maketitle

\abstract*{We consider the stochastically perturbed cubic difference equation with variable coefficients 
\begin{equation*}
x_{n+1}=x_n(1-h_nx_n^2)+\rho_{n+1}\xi_{n+1}, \quad n\in \mathbb N,\quad  x_0\in \mathbb R.
\end{equation*}
Here $(\xi_n)_{n\in \mathbb N}$ is a sequence of independent random variables, and $(\rho_n)_{n\in \mathbb N}$  and  $(h_n)_{n\in \mathbb N}$ are sequences of nonnegative real numbers. We can stop the  sequence  $(h_n)_{n\in \mathbb N}$ 
after some random time  $\mathcal N$ so it becomes a constant sequence, where the common value is an $\mathcal{F}_\mathcal{N}$-measurable random variable. We derive conditions on the sequences $(h_n)_{n\in \mathbb N}$, $(\rho_n)_{n\in \mathbb N}$  and $(\xi_n)_{n\in \mathbb N}$,  which guarantee that $\lim_{n\to \infty} x_n$ exists almost surely (a.s.), and that the limit is equal to zero a.s. for any initial value $ x_0\in \mathbb R$.
}

\abstract{We consider the stochastically perturbed cubic difference equation with variable coefficients 
\[
x_{n+1}=x_n(1-h_nx_n^2)+\rho_{n+1}\xi_{n+1}, \quad n\in \mathbb N,\quad  x_0\in \mathbb R.
\]
Here $(\xi_n)_{n\in \mathbb N}$ is a sequence of independent random variables, and $(\rho_n)_{n\in \mathbb N}$  and  $(h_n)_{n\in \mathbb N}$ are sequences of nonnegative real numbers. We can stop the  sequence  $(h_n)_{n\in \mathbb N}$ 
after some random time  $\mathcal N$ so it becomes a constant sequence, where the common value is an $\mathcal{F}_\mathcal{N}$-measurable random variable. We derive conditions on the sequences $(h_n)_{n\in \mathbb N}$, $(\rho_n)_{n\in \mathbb N}$  and $(\xi_n)_{n\in \mathbb N}$,  which guarantee that $\lim_{n\to \infty} x_n$ exists almost surely (a.s.), and that the limit is equal to zero a.s. for any initial value $ x_0\in \mathbb R$.}

\section{Introduction}
In this paper we analyse the global almost sure (a.s.) asymptotic behaviour of solutions of
a cubic difference equation with variable coefficients and subject to stochastic perturbations
\begin{equation}
\label{eq:stochintr}
x_{n+1}=x_n(1-h_nx_n^2)+\rho_{n+1}\xi_{n+1}, \quad n\in \mathbb N,\quad  x_0\in \mathbb R.
\end{equation}
Here $(\xi_n)_{n\in \mathbb N}$ is a sequence of independent identically distributed random variables, $(\rho_n)_{n\in \mathbb N}$ is a sequence of nonnegative reals, and $(h_n)_{n\in \mathbb N}$ is a sequence of nonnegative reals.

When $(\xi_n)_{n\in \mathbb N}$ is an independent sequence of standard Normal random
variables, \eqref {eq:stochintr} can be interpreted as the Euler-Maruyama discretisation of the It\^o-type
stochastic differential equation
\begin{equation}
\label{eq:Ito}
dX_t=-bX_t^3dt+g(t)dW_t, \quad n\in \mathbb N,\quad  X_0\in \mathbb R,
\end{equation}
where $(W_t)_{t\ge 0}$ is a standard Wiener process, $b>0$ is some constant, $g:[0, \infty)\to [0, \infty)$ is  a continuous function. It was shown in \cite{CW} that when $\lim_{t\to \infty}\rho^2(t)\ln t=0$ solutions of stochastic differential equation \eqref{eq:Ito} are globally a.s. asymptotically stable, i.e.  $\lim_{t\to \infty} X_t=0$, almost surely,  for any initial value $ X_0\in \mathbb R$.

There is an extensive literature on the global a.s. asymptotic behaviour of solutions of nonlinear stochastic difference
equations, and the most relevant publications for our purposes are: [1, 2, 3, 4, 5, 7, 14, 15]. However, if the timestep sequence in Eq. \eqref{eq:stochintr} is constant, so that $h_n\equiv h$, the global dynamics of \eqref{eq:Ito} are not preserved and convergence of solutions to zero will only occur on a restricted subset of initial values. An early attempt to address local dynamics in an equation with bounded noise can be found in \cite{F}; general results for equations with fading, state independent noise may be found in \cite{ABR}. In \cite{AKMR} a complete description is given of these local dynamics (see also \cite{ABR} and \cite{AMcR}). It was proved that the set of initial values can be partitioned into a ``stability'' region, within which solutions converge asymptotically to zero, an ``instability'' region, within which solutions rapidly grow without bound, and a region of unknown dynamics that is in some sense small. In the first two cases, the dynamic holds with probability at least $1-\gamma$ for $\gamma \in (0, 1)$. 

In the same article, it was shown that for any initial value $x_0\in \mathbb R$, the behaviour of solution of the difference equation can be made consistent with the corresponding solution of the differential equation, with probability $1-\gamma$, by choosing the stepsize parameter $h$ sufficiently small. This observation motivates the approach taken in this article, wherein the stepsize parameter is allowed to decrease over a random interval in order to capture trajectories within the basin of attraction of the point at zero long enough to ensure asymptotic convergence. 

Several recent publications are devoted to the use of adaptive timestepping in a explicit Euler-Maruyama discretization of nonlinear equations: for example \cite{AKR2010,KL2017,LM,KRR}. In \cite{KL2017} (see also \cite{Giles2017}) it was shown that suitably designed adaptive timestepping strategies could be used to ensure strong convergence of order $1/2$ for a class of equations with non-globally Lipschitz drift, and globally Lipschitz diffusion. These strategies work by controlling the extent of the nonlinear drift response in discrete time and required that the timesteps depend on solution values. In \cite{KRR} an extension of that idea allows an explicit Euler-Maruyama discretisation to reproduce dynamical properties of a class of nonlinear stochastic differential equations with a unique equilibrium solution and non-negative, non-globally Lipschitz drift and diffusion coefficients. The a.s. asymptotic stability and instability of the equilibrium at zero is closely reproduced, and positivity of solutions is preserved with arbitrarily high probability. 

An element that these articles have in common is that the variable time-step depends upon the value of the solution. By contrast, in the present paper the sequence $(h_n)_{n\in \mathbb N}$ does not, and will be the same for any given initial value $x_0\in \mathbb R$.
However since the values of $h_n$ can become arbitrarily small, it is not necessarily the case that $x_n$ converges to zero: in fact if the stepsize sequence is summable we will show that the limit is necessarily nonzero a.s. So we freeze the sequence $(h_n)_{n\in \mathbb N}$ at an appropriate random moment $\mathcal N$, i.e. all step-sizes after $\mathcal N$ are the same: $h_n=h_{\mathcal N}$ for $n\ge \mathcal N$. The time at which this occurs depends on the initial value $x_0$, and is chosen to ensure that $x_n$ converges to zero a.s., as required. 

The structure of the article is as follows. Some necessary technical results are stated in Section \ref{sec:prelim}. In Section \ref{sec:homdet} we construct a timestep sequence $(h_n)_{n\in\mathbb{N}}$ that ensures solutions of the unperturbed cubic difference equation converge to a finite limit, and show that the summability of $(h_n)_{n\in\mathbb{N}}$ determines whether or not that limit is zero. In Section \ref{sec:nohomdet} we examine the convergence of solutions under the influence of a deterministic perturbation, and in Section \ref{sec:randompert} we consider two kinds of stochastic perturbation; one with bounded noise, and one with Gaussian noise. Illustrative numerical examples are provided in Section \ref{sec:exsim}.

\section{Mathematical preliminaries}\label{sec:prelim}
Everywhere in this paper, let  $(\Omega, {\mathcal{F}}, {\mathbb{P}})$ be  a complete probability space. A detailed discussion of probabilistic  concepts  and notation may be found, for example, in Shiryaev~\cite{Shiryaev96}. We will use the following elementary inequality: for each $a,b >0$  and $\alpha\in (0, 1)$ 
\begin{equation}
\label{ineq:alpha}
(a+b)^{\alpha}\le a^{\alpha}+b^{\alpha}.
\end{equation}

The following lemmas also present additional useful technical results:
\begin{lemma}
\label{lem:monf}
Let $f:[0, \infty)\to [0, \infty)$ be a decreasing continuous function, then
\[
 \int_0^{n+1} f(x)dx>\sum_{i=1}^{n}f(i) > \int_1^{n+1} f(x)dx>\sum_{i=2}^{n+1}f(i).
\]
\end{lemma}

\begin{lemma}
\label{lem:logest}
\be
\item [(i)] $\ln (1-x)<-x$ for $-\infty<x<0$;
\item [(ii)] For $0<x<\frac{1}{2}$ the following estimate holds
\begin{equation}
\label{est:log}
\ln(1-x)>-2x.
\end{equation}
\ee
\end{lemma}

\begin{lemma}
\label{lem:product}
Let $q_n\in [0, 1)$ for all $n\in \mathbb N$. Then $\prod_{n=1}^\infty (1-q_n)$ converges to non zero limit if and only if $\sum_{n=1}^\infty q_n$ converges.
\end{lemma}
We adopt the convention $\displaystyle \prod_{n=i}^j1=1$ if $i>j$ from here forwards. The next result can be found in \cite[Ch. 4.4, Ex. 1]{Shiryaev96}.
\begin{lemma}
\label{lem:norm}
Let $(\xi_n)_{n\in \mathbb N}$ be a sequence of independent  $\mathcal N(0, 1)$ distributed random variables. Then
\begin{equation}
\label{lim:norm}
\mathbb P \left\{\limsup_{n\to \infty}\frac {\xi_n}{\sqrt{2\ln n}} =1 \right\}=1.
\end{equation}
\end{lemma}
We will use the following notation throughout the article:
\begin{definition}
Denote, for $k\in \mathbb N$, 
\begin{equation}
\label{def:eklnk}
\begin{split}
&e_{[k]}^a=\exp\{\exp\{\dots \{a\underbrace{\}\dots\}}_{k \,\,  times}\quad \mbox{for each} \quad  a\in \mathbb R, \quad e_{[0]}^a=1;\\
 &\ln_k b= \ln[ \ln [\dots[\ln b\underbrace{] \dots] ]}_{k \,\,  times} \quad \mbox{for each} \quad  b\ge  e_{[k]}^1, \quad \ln _0 b=b.
\end{split}
\end{equation}
\end{definition}
\begin{corollary}
\label{cor:lnk}
For all $n,k\in\mathbb{N}$,
\begin{multline}
\label{est:lnk1}
\sum_{i=j}^{n} \frac 1{(i+1)\ln (i+1)\dots \ln_k \left(i+e_{[k]}^1\right)}\\>\int_j^{n+1} \frac {dy}{(y+e_{[k]}^1)\ln (y+e_{[k]}^1)\dots \ln_k \left(y+e_{[k]}^1\right)}\\
=\ln_{k+1} (n+1+e_{[k]}^1)-\ln_{k+1} (j+e_{[k]}^1),
\end{multline}

and 
\begin{multline}
\label{est:lnk2}
\sum_{i=1}^{n+1}\frac 1{\left(i+e_{[k]}^1\right)\ln \left(i+e_{[k]}^1\right)\dots \ln_k \left(i+e_{[k]}^1\right)}\\<\int_0^{n+2} \frac {dy}{(y+e_{[k]}^1)\ln (y+e_{[k]}^1)\dots \ln_k \left(y+e_{[k]}^1\right)}\\
=\ln_{k+1} (n+2+e_{[k]}^1)-\ln_{k+1} (e_{[k]}^1)=\ln_{k+1} (n+2+e_{[k]}^1).
\end{multline}
\end{corollary}
\begin{proof}
Applying Lemma \ref{lem:monf} to the decreasing, continuous function  $$f(x)=\frac 1{(x+1)\ln (x+1)\dots \ln_k \left(x+e_{[k]}^1\right)}$$ yields the result.
\end{proof}

\section{The unperturbed deterministic cubic equation}
\label{sec:homdet}
Consider
\begin{equation}
\label{eq:main3hom}
x_{n+1}=x_n(1-h_nx_n^2), \quad x_0\in \mathbb R, \quad n\in \mathbb N.
\end{equation}
Everywhere in this paper we assume that $(h_n)_{n\in\mathbb N}$ is a non-increasing sequence of positive numbers.
We derive an estimate on each $|x_n|$ and present a time-step sequence $(h_n)_{n\in\mathbb{N}}$ which provides convergence of the solution for any initial value $x_0\in \mathbb R$.

\subsection{Preliminary lemmata on solutions of Eq. \eqref{eq:main3hom}}

\begin{lemma}
\label{lem:1}
Let $x_n$ be a solution to equation \eqref{eq:main3hom}. Assume that \begin{equation}
\label{ass:1}
\text{ there exists  $N\in \mathbb N$ such that $h_Nx_N^2<2$.}
\end{equation}
Then,
\begin{enumerate}
\item [(a)]  the sequence $ (|x_n| )_{n\in \mathbb N}$ is non-increasing and $h_nx_n^2<2$ for each $n\ge N$;
\item[(b)]   the sequence $ (|x_n| )_{n\in \mathbb N}$ converges to a finite limit.
\end{enumerate}
\end{lemma}

\begin{proof}
\emph{(a)}  Since $h_Nx_N^2<2$ implies that $1-h_Nx_N^2 \in (-1,1)$ we have 
\begin{equation}
\label{rel:NN+2}
|x_{N+1}|=|x_N||1-h_Nx^2_N|<|x_N|.
\end{equation}
Since $(h_n)_{n\in \mathbb N}$ is a non-increasing sequence,  we have  $h_N\ge h_{N+1}$ and  
\[
h_{N+1}x^2_{N+1}<h_Nx^2_N<2.
\]
The remainder of the proof of (a) follows by induction. To prove (b) we note that  the sequence $(|x_n|)_{n\in\mathbb{N}}$ is non-increasing and bounded below by $0$, and therefore it converges to a finite limit.
 \end{proof}

\begin{lemma}
\label{lem:2}
Let $(x_n)_{n\in\mathbb{N}}$ be a solution to equation \eqref{eq:main3hom}. Assume  that there exist $N\in \mathbb N$ such that
\begin{equation}
\label{ass:2}
2>h_Nx_N^2>1.
\end{equation}
Then there exists $N_1 > N$ such that $h_{N_1}x_{N_1}^2 \leq 1$.
\end{lemma}

\begin{proof}
By Lemma \ref{lem:1}, the sequence $(|x_n|)_{n\in\mathbb{N}}$ is non-increasing.  Furthermore, Lemma \ref{lem:1} part (b) implies that, for some $L\in\mathbb R$,
\begin{equation}
\label{lim:L^2}
\lim_{n\to \infty} x_n^2=L^2.
\end{equation}

Proceed by contradiction and assume that $h_nx^2_n >1$ for all $n\ge  N$.  If either $L=0$ or $\lim_{n\to \infty}h_n=0$, it follows that $\lim_{n\to\infty}h_nx^2_n=0$.  So $L\neq 0$ and  $\lim_{n\to \infty}h_n=K\neq 0$. 
Since  $h_nx_n^2$ is not increasing and by \eqref{ass:2} we have 
\[
1\le L^2K<h_nx^2_n<2.
\]
So it is only possible that either 
\begin{enumerate}[(i)]
\item $2>L^2K>1$ or 
\item $L^2K=1$.
\end{enumerate}
For case (i), $1-L^2K \in (-1,0)$.  Since $\lim_{n\to \infty}h_nx^2_n= L^2K$, there exists $\delta\in (0, 1)$ and $N_1\in \mathbb N$ such that
$|1- h_nx^2|<\delta$, for all $n\ge N_1$, implying
\begin{equation}
\label{ineq:d}
 |x_{n+1}|<\delta |x_n|,\quad n\geq N_1.
 \end{equation}

Passing to the limit of both sides of \eqref{ineq:d} as $n\to\infty$, we get $L < \delta L$. Since $\delta\in(0,1)$, case (i) leads to a contradiction.

For case (ii), we have 
\[
\lim_{n \to \infty}|x_{n+1}|=\lim_{n \to \infty}|x_n|\lim_{n \to \infty}|1-h_nx_n^2| =0,
\]
which implies that $\lim_{n \to \infty} |x_n| = 0$.  Hence, case (ii) also leads to a contradiction. This completes the proof.
\end{proof}

\begin{lemma}
\label{lem:N}
Let $(x_n)_{n\in\mathbb{N}}$ be a solution to \eqref{eq:main3hom} with arbitrary initial condition $x_0 \neq 0$.  If
\begin{equation}
\label{ineq:<}
\text{ there exists  $N\in \mathbb N$ such that $h_Nx_N^2<1$},
\end{equation}
then
\begin{enumerate}
\item [(a)] terms of the sequence $(x_n)_{n\geq N}$ do not change sign;
\item  [(b)]  the sequence $(x_n)_{n\in \mathbb N}$ converges to a finite limit.
\end{enumerate}
\end{lemma}

\begin{proof}
(a) Since \eqref{ineq:<} implies \eqref{ass:1}, we conclude that the sequence $(|x_n|)_{n\in \mathbb N}$ is non-increasing and therefore convergent, $1-h_nx_n^2\in (0, 1)$ for all $n\ge N$ and then  $x_n x_{n+1}>0$ for all $n\ge N$. So  the sign of $x_n$ stops changing for $n\ge N$, which implies that the sequence $(x_n)_{n\in \mathbb N}$ converges to a finite limit. 
 \end{proof}
 
\begin{remark}
\label{rem:<2}
From Lemma \ref{lem:2} we conclude  that condition \eqref{ass:1} implies \eqref{ineq:<}. So without loss of generality we refer to \eqref{ineq:<} instead of \eqref{ass:1} for the remainder of the article.
\end{remark}

\begin{remark}
\label{rem:=1}
In  the case where $h_N x_N^2 = 1$ for some $N \in \mathbb N$,  we have $x_n=0$ for all $n>N$,  ensuring that $\lim_{n\to \infty} x_n = 0$. In  the case when $h_N x_N^2 = 2$ for some $N \in \mathbb N$, we have 
\[
x_{N+1}=x_N(1-h_Nx^2_N)= -x_N, 
\] 
which implies that $ x_{N+k}=(-1)^k x_N$. In this case $\lim_{n\to \infty} |x_n| = |x_N|$ but $\lim_{n\to \infty} x_n$ does not exist.
\end{remark}

\subsection{Timestep summability and the limit of solutions}
\label{subsec:con}

In this section we show that if \eqref{ineq:<} holds, then solutions converge to a nonzero limit if the stepsize sequence is summable. If not, solutions converge asymptotically to zero.
\begin{lemma}
\label{lem:finite}
Let $(x_n)_{n\in\mathbb{N}}$ be a solution of \eqref{eq:main3hom} with initial condition $x_0 \neq 0$.  Suppose that \eqref{ineq:<} holds and that $\sum_{j=1}^{\infty}h_{j}=S < \infty$.  Then, $\lim_ {n \to \infty} x_n =L \neq 0$.
\end{lemma}

\begin{proof}
 Since \eqref{ineq:<} holds for some $N \in \mathbb N$, by Lemmata \ref{lem:1} and  \ref{lem:N} we have, for all $k \in \mathbb N$,
\begin{equation}
\label{ineq:k}
x^2_{N+k}<  x_N^2, \quad 1-h_{N+i} x_{N+i}^2>0.
\end{equation}
Then, for all $k \in \mathbb N$,
\begin{equation*}
%\label{cond:in1}
\begin{split}
x_{N+k} =& x_{N+k-1}(1-h_{N+k-1} x_{N+k-1}^2) \\&= x_{N+k-2}(1-h_{N+k-2} x_{N+k-2}^2)(1-h_{N+k-1} x_{N+k-1}^2)\\&
= x_N \prod_{i=0}^{k-1} \left(1-h_{N+i} x_{N+i}^2 \right).   
\end{split}
\end{equation*}
This implies 
\begin{equation}
\label{eq:2.1}
x_{N+k} = x_N e^{\sum_{i=0}^{k-1} \ln\left(1-h_{N+i} x_{N+i}^2 \right)}.
\end{equation}
By Lemma \ref{lem:1}, part (a),
\[
\sum_{i=0}^{k-1} h_{N+i} x_{N+i}^2 < x_N^2 \sum_{i=0}^{k-1}h_{N+i} < x_N^2 S.
\]
By  Lemma \ref{lem:N}, part (b),  for some $L\in \mathbb R$ we have $\lim_ {n \to \infty} x_n = L$. Also, $\lim_{j \to \infty} h_j = 0$, since $\sum_{j=1}^{\infty}h_{j} < \infty$.  So there exists $N_1 \in \mathbb N$  such that $h_{n}x^2_{n} < \frac{1}{2}$ for all $n\ge N_1$. Without loss of generality we may therefore suppose that $N_1=N$. Part (ii) of Lemma \ref{lem:logest} applies, and so for all $i \in \mathbb N$,
\begin{equation}
\label{est>}
\ln\left(1-h_{N+i} x_{N+i}^2 \right) > -2h_{N+i} x_{N+i}^2.
\end{equation}
Let $x_N >0$.  By applying \eqref{est>} to \eqref{eq:2.1}, and by \eqref{ineq:k}, we have
\[
x_{N+k}  > x_N e^{-2\sum_{i=0}^{k-1} h_{N+i} x_{N+i}^2}\ge  x_N e^{-2x_N^2\sum_{i=0}^{k-1} h_{N+i}} > x_N e^{-2x_N^2S}>0.
\]
Passing to the limit for $k\to \infty$ in above inequality we get 
\[
L=\lim_{n \to \infty}x_{n} > x_N e^{-2x_N^2S}>0.
\]
Similarly, for $x_N <0$, we have
\[
x_{N+k}  < x_N e^{-2\sum_{i=0}^{k-1} h_{N+i} x_{N+i}^2}\le  x_N e^{-2x_N^2\sum_{i=0}^{k-1} h_{N+i}} < x_N e^{-2x_N^2S}<0.
\]
In both cases $\lim_{n \to \infty} x_{n} \neq 0$, proving the statement of the Lemma.
\end{proof}

\begin{lemma}
\label{lem:infinite}
Let $(x_n)_{n\in\mathbb{N}}$ be a solution to  \eqref{eq:main3hom}  with the  initial value $x_0 \neq 0$.  Suppose that \eqref{ineq:<} holds and that $\sum_{j=1}^{\infty}h_{j} = \infty$ . Then $\lim_ {n \to \infty} x_n = 0$.
\end{lemma}

\begin{proof}
First, \eqref{ineq:<} implies  \eqref{ineq:k}.  So, by Lemma \ref{lem:logest} part (i), for each $k \in \mathbb N$,
\begin{equation}
\label{approx:est} 
\ln(1-h_{N+k} x_{N+k}^2) < -h_{N+k} x_{N+k}^2.
\end{equation}
Proceed by contradiction, and suppose that $\lim_{n \to \infty} x_n^2 =L^2$ for some $L>0$.  Since  the sequence $(|x_n|)_{n\in \mathbb N}$ is non-increasing, we have  $x_N^2>x_{N+i}^2 \geq L^2$.  Applying \eqref{eq:2.1} and \eqref{approx:est} we obtain
\begin{equation}
\label{est:2.1}
\begin{split}
|x_{N+k}| &= |x_N|e^{\sum_{i=0}^{k-1} \ln\left(1-h_{N+i} x_{N+i}^2 \right)} < |x_N| e^{\sum_{i=0}^{k-1} \left(-h_{N+i} x_{N+i}^2 \right)} \\&< |x_N| e^{-L^2\sum_{i=0}^{k-1} h_{N+i}}.
\end{split}
\end{equation}
Passing to the limit in \eqref{est:2.1} as $k\to \infty$, we arrive at 
\[
L^2<|x_N|  e^{-L^2\sum_{j=1}^{\infty}h_j}=0,
\]
yielding the desired contradiction.
\end{proof}

%%%%%%%%%%%%%%%%%%%%%%%%%%%%%%%%%%%%
\subsection{Estimation of $|x_n|$}
\label{subsec:est|xn|}

In this section we establish a useful estimate for each $|x_n|$ when there exists $\bar N\in \mathbb N$ such that 
\begin{equation}
\label{cond:in1}
\frac 1{h_nx_n^2}\in (0, 1),\text{ for all  }n\le \bar N.
\end{equation}

\begin{lemma}
\label{lem:estup}
If \eqref{cond:in1} holds for some $\bar N\in \mathbb N$,  then for all $n\le \bar N$
\begin{equation}
\label{est:main}
 |x_n| < |x_0|^{{3}^{n}}\prod_{i=0}^{n-1} h_{n-1-i}^{{3}^{i}},  \quad n \in \mathbb N.
\end{equation}
\end{lemma}

\begin{proof}
For $n=0$ we have, 
\[
x_1=x_0(1-h_0x_0^2)=-h_0x_0^3 \left(1-\frac 1{h_0x_0^2}  \right),
\]
which,  by \eqref{cond:in1}, implies  that
\[
|x_1|=\left|h_0x_0^3 \left(1-\frac 1{h_0x_0^2}  \right)   \right|= h_0|x_0|^3 \left|1-\frac 1{h_0x_0^2}\right|
< h_0|x_0|^3.
\]
So \eqref{est:main} holds for $n=1$. Assume that \eqref{est:main} holds  for some $k<\bar N$. By \eqref{cond:in1},  
$|x_{k+1}| < h_k|x_k|^3,$  which implies that 
\[
|x_{k+1}| < h_k |x_k|^3< h_k |x_0|^{3^{k+1}}\prod_{i=0}^{k-1} h_{k-1-i}^{{3}^{i+1}} = |x_0|^{3^{k+1}} \prod_{i=-1}^{k-1} h_{k-1-i}^{{3}^{i+1}},
\]
which demonstrates that \eqref{est:main} holds for $k+1$, and concludes the proof for all $n\leq \bar N$ by induction. 
\end{proof}

%%%%%%%%%%%

\begin{lemma}
\label{lem:he}
Let $(x_n)_{n\in\mathbb{N}}$ be a solution to \eqref{eq:main3hom} with arbitrary $x_0 \in \mathbb R$ and with $(h_n)_{n\in\mathbb{N}}$ satisfying the following condition
\begin{equation}
\label{cond:loginfty} 
\sum_{j=0}^{\infty} 3^{-j}\ln h^{-1}_{j}=\infty.
\end{equation}
Then there exists $\bar N=\bar N(x_0)$ such that \eqref{ineq:<} holds.
\end{lemma}

\begin{proof} 
Suppose that  \eqref{ineq:<} fails to hold for any $\bar N$. Then  $1/{h_nx_n^2}\in (0, 1)$, for all $n\in\mathbb{N}$.  For an arbitrary $\bar N$, we can apply Lemma \ref{lem:estup}, making the change of variables 
\[
j=\bar N-1-i, \quad i=\bar N-1-j, \quad i=0, \dots, \bar N-1, \quad j=\bar N-1,  \dots, 0,
\]
to get
\begin{equation}
\label{ineq+1}
 |x_{\bar N}| <  |x_0|^{{3}^{\bar N}}\prod_{j=0}^{\bar N-1} h_{j}^{{3}^{\bar N-1-j}}.
\end{equation}
Set 
\begin{equation}
\label{est: barN}
F(\bar N) : = h_{\bar N} \left| |x_0|^{{3}^{\bar N}}\prod_{j=0}^{\bar N-1} h_{j}^{{3}^{\bar N-1-j}}\right|^2.
\end{equation}
Squaring both sides of \eqref{ineq+1} and multiplying throughout by $h_{\bar N}$, we obtain $ h_{\bar N} |x_{\bar N}|^2<F(\bar N)$.
Then
\begin{equation}
\label{log}
\begin{split}
\ln \left[F(\bar N)\right] &=\ln h_{\bar N} + 2 \cdot 3^{\bar N} \ln |x_0| +\sum_{j=0}^{\bar N-1} \ln h_{j}^{2 \cdot 3^{\bar N-1-j}}
\\&=  \ln h_{\bar N} + 2 \cdot 3^{\bar N} \ln |x_0| + \frac{2}{3} \cdot \sum_{j=0}^{\bar N -1} 3^{\bar N-j}\ln h_{j}.
%\\&=   2 \cdot 3^{\bar N} \ln |x_0| + \frac{2}{3} \cdot 3^{\bar N}\sum_{j=0}^{\bar N} 3^{-j}\ln h_{j}.
\end{split}
\end{equation}
Without loss of generality we can assume that $\ln h_{\bar N}<0$, so $\ln h_{\bar N}<\frac 23\ln h_{\bar N}$ and, continuing from \eqref{log},
\begin{equation}
\label{log1}
\begin{split}
\ln \left[F(\bar N)\right]&\le 2 \cdot 3^{\bar N} \ln |x_0| + \frac{2}{3} \cdot 3^{\bar N}\sum_{j=0}^{\bar N} 3^{-j}\ln h_{j}
\\&=   \frac 23 \cdot 3^{\bar N} \left[\ln |x_0|^3 + \sum_{j=0}^{\bar N} 3^{-j}\ln h_{j}\right].
\end{split}
\end{equation}
The expression in the square brackets is negative for any $x_0\in \mathbb R$ with $\bar N$ sufficiently large if condition \eqref{cond:loginfty} holds. In this case  for each $x_0\in \mathbb R$ we can find $\bar N=\bar N(x_0)$ s.t. 
\[
\sum_{j=0}^{\bar N} 3^{-j}\ln h^{-1}_{j}>\ln |x_0|^3.
\]
Then $F(\bar N)<1$ which means that $|x_{\bar N}|<1$  as well as $h_{\bar N}x_{\bar N}^2<1$. So  condition \eqref{ineq:<} holds for $\bar N=\bar N(x_0)$. Obtained contradiction proves the result.

\end{proof}

Lemmata \ref{lem:infinite} and \ref{lem:he} imply the following corollary. 

\begin{corollary}
\label{cor:convhom}
Let $(x_n)_{n\in\mathbb{N}}$ be a solution to \eqref{eq:main3hom} with arbitrary $x_0 \in \mathbb R$ and with $(h_n)_{n\in\mathbb{N}}$ satisfying condition
\eqref{cond:loginfty}. Then  $\lim_ {n \to \infty} x_n=0$.
\end{corollary}

\begin{lemma}
\label{lem:hepartial}
Condition \eqref {cond:loginfty} holds if 
\be 
\item[(i)] $h_n\le e^{-3^n}$;
\item[(ii)] $h_n\le e^{-\frac{3^n}n}$;
\item[(iii)] $h_n\le e^{-\frac{3^n}{n\ln n}}$;
\item[(iv)] $h_n\le e^{-\frac{3^n}{n\ln n \ln_2 n \dots \ln_k n}}$.
\ee

\end{lemma}

\begin{proof}
Case (i): we have  $3^{-j}\ln h^{-1}_{j}\geq 1$. Case (ii):  we have  $3^{-j}\ln h^{-1}_{j}\ge \frac 1j$. Cases (iii) and (iv):  we have  $3^{-j}\ln h^{-1}_{j}\ge \frac 1{j\ln j}, \dots$ {\it etc}.  Note that the series 
\[
\sum^\infty 3^{-j}\ln h^{-1}_{j}=\infty,
\]
for $h_j$ defined by each of (i)-(iv). The lower limit of summation should be chosen according to $h_j$ in order to avoid zero denominators.
\end{proof}

\begin{remark}
\label{lrem:hepartial}
Applying Lemma \ref{lem:monf} we conclude that
for $h_j$ defined by each of (i)-(iv), the corresponding $\bar N(x_0)$ can be estimated as
\be 
\item[(i)] $\bar N(x_0)>\ln |x_0|^3$;
\item[(ii)] $\ln \bar N(x_0)>\ln |x_0|^3$, so $\bar N(x_0)>|x_0|^3$;
\item[(iii)] $\ln[ \ln[\bar N(x_0)]]>\ln |x_0|^3$, so $\bar N(x_0)>e^{|x_0|^3}$;
\item[(iv)]  $\ln_{k-1}[\bar N(x_0)]>\ln |x_0|^3$, so $\bar N(x_0)>e^{e^{\dots^{|x_0|^3}}}$.
\ee
\end{remark}

\section{The perturbed deterministic cubic difference equation}
\label{sec:nohomdet}
Consider the perturbed difference equation
\begin{equation}
\label{eq:main3}
x_{n+1}=x_n(1-h_nx_n^2)+u_{n+1}, \quad x_0\in \mathbb R.
\end{equation}
where $(u_n)_{n\in\mathbb{R}}$ is a real-valued sequence. We begin by providing an estimate for solutions of \eqref{eq:main3} under condition \eqref{cond:in1}.
\begin{lemma}
\label{lem:estxn+1}
Let $(x_n)_{n\in\mathbb{N}}$ be a solution to equation \eqref{eq:main3} and let condition \eqref{cond:in1} hold. Then, for $n\le \bar N$,
\begin{equation}
\label{est:mainonhomn}
\begin{split}
 |x_{n+1}|^{\frac 1{3^{n+1}}}& < |x_0|\prod_{i=0}^{n} h_{i}^{\frac 1{3^{i+1}}}+\sum_{i=1}^{n+1}\prod_{j=i}^{n} h_{j}^{\frac 1{3^{j+1}}}| u_{i}|^{\frac 1{3^{i}}}\\&=\prod_{i=0}^{n} h_{i}^{\frac 1{3^{i+1}}}\left[ |x_0|+\sum_{i=1}^{n+1}\prod_{j=0}^{i-1} h_{j}^{-\frac 1{3^{j+1}}}| u_{i}|^{\frac 1{3^{i}}}\right].
\end{split}
\end{equation}
\end{lemma}

\begin{proof}
By  condition \eqref{cond:in1},  for each $n\le \bar N$ we have
\begin{equation}
\label{est:n+1}
\begin{split}
|x_{n+1}|&\le |x_n(1-h_nx_n^2)|+|u_{n+1}|\\&\le h_n|x_n|^3 \left |1-\frac 1{h_nx_n^2}\right|+|u_{n+1}|\\&\le  h_n|x_n|^3 +|u_{n+1}|.
\end{split}
\end{equation}
Applying the inequality \eqref{ineq:alpha} with $\alpha_1=\frac 13$, to \eqref{est:n+1} with $n=0$, we get
\begin{equation}
\label{est:2}
|x_1|^{\frac 13}\le h_0^{\frac 13}|x_0| +|u_{1}|^{\frac 13}.
\end{equation}
Applying the inequality \eqref{ineq:alpha} with $\alpha_2=\frac 1{3^2}$,  to \eqref{est:n+1} with $n=1$, and substituting  \eqref{est:2}, we get
\begin{equation}
\label{est:3}
|x_2|^{\frac 1{3^2}}\le h_1^{\frac 1{3^2}}|x_1|^{\frac 13} +|u_{2}|^{\frac 1{3^2}}\le  h_1^{\frac 1{3^2}} h_0^{\frac 13}|x_0|+h_1^{\frac 1{3^2}} |u_{1}|^{\frac 13}+|u_{2}|^{\frac 1{3^2}}.
\end{equation}
Continue this process inductively,  and applying  the inequality \eqref{ineq:alpha} with $\alpha_n=\frac 1{3^{n+1}}$ we get
\begin{equation*}
\begin{split}
|x_{n+1}|^{\frac 1{3^{n+1}}}&\le h_{n}^{\frac 1{3^{n+1}}} h_{n-1}^{\frac 1{3^{n}}}\dots h_1^{\frac 1{3^2}} h_0^{\frac 13}|x_0|+  h_{n}^{\frac 1{3^{n+1}}}h_{n-1}^{\frac 1{3^{n}}}\dots h_2^{\frac 1{3^2}} h_1^{\frac 13} |u_{1}|^{\frac 13}\\&+ h_{n}^{\frac 1{3^{n+1}}}h_{n-1}^{\frac 1{3^{n}}}\dots h_3^{\frac 1{3^4}}h_2^{\frac 1{3^3}} |u_{2}|^{\frac 1{3^2}}+\dots+  h_{n}^{\frac 1{3^{n+1}}}|u_{n}|^{\frac 1{3^n}} +|u_{n+1}|^{\frac 1{3^{n+1}}},
\end{split}
\end{equation*}
which completes the proof.
\end{proof}
\subsection{Boundedness of $(|x_n|)_{n\in\mathbb{N}}$ for particular $(h_n)_{n\in\mathbb{N}}$ and $(u_n)_{n\in\mathbb{N}}$}
\label{sec:boundxn}

In this section we consider two special cases of $h_n$ and $u_n$ each of which guarantees the boundedness of the sequence $(|x_n|)_{n\in \mathbb N}$. Both forms of $h_n$ were introduced in Lemma \ref{lem:hepartial}: the first corresponds to (ii)- (iv), the second corresponds to (i).  Estimates of $|u_n| $ are chosen relative to the estimates for $|h_n|$.

\subsubsection{Case 1}
\label{subsub:e3lnn}
Let $e_{[k]}^1$ and $\ln_k(\cdot)$ be defined as in \eqref{def:eklnk}.
 Assume that, there exists  $k\in \mathbb N$ and $\beta\in (0, 1)$ such that
\begin{equation}
\label{def:h1}
\begin{split}
&h_n\le \exp\left\{{-\frac{3^{n+1}}{\left(n+e_{[k]}^1\right)\ln \left(n+e_{[k]}^1\right)\dots \ln_k \left(n+e_{[k]}^1\right)}}\right\},  \\&(h_n)_{n\in \mathbb N} \quad \mbox{is a decreasing sequence},
\end{split}
\end{equation}
and 
\begin{equation}
\label{def:u1}
 |u_n|:\le \left( \frac \beta{\left(n+e_{[k]}^1\right)\ln \left(n+e_{[k]}^1\right)\dots \ln_k \left(n+e_{[k]}^1\right)}\right)^{3^{n}}.
\end{equation}

\begin{lemma}
\label{lem:N1}
Let $(x_n)_{n\in\mathbb{N}}$ be a solution to equation \eqref{eq:main3} and let $(h_n)_{n\in\mathbb{N}}$ and $(u_n)_{n\in\mathbb{N}}$ satisfy \eqref{def:h1} and  \eqref{def:u1}, respectively. Then 
\be
\item [(i)] there exists $N_1$ such that $|x_{N_1+1}|<1$, and \eqref{ineq:<} holds;
\item [(ii)] $|x_{N_1+i}|$ is uniformly bounded for all $i\in \mathbb N$.
\ee
\end{lemma}
\begin{proof}
Suppose to the contrary that \eqref{cond:in1} holds for all $n$. Then, by Lemma \ref{lem:estxn+1}, estimate \eqref{est:mainonhomn} holds for all $n\in \mathbb N$. 

Substituting the values of $h_n$ from \eqref{def:h1} and $u_n$ from \eqref{def:u1}
into \eqref{est:mainonhomn} we get 
\begin{equation*}
\begin{split}
&|x_{n+1}|^{\frac 1{3^{n+1}}}\le \exp\left\{-\sum_{i=0}^{n} \frac 1{\left(i+e_{[k]}^1\right)\ln \left(i+e_{[k]}^1\right)\dots \ln_k \left(i+e_{[k]}^1\right)}\right\}|x_0|\\&+ 
\sum_{i=1}^{n+1}\exp\left\{-\sum_{j=i}^{n}\frac 1{(j+e_{[k]}^1)\ln (j+e_{[k]}^1)\dots \ln_k \left(n+2+e_{[k]}^1\right)}\right\}|u_j|^{\frac1{3^i}}.
\end{split}
\end{equation*}
Now we apply the inequalities from \eqref{est:lnk1} and \eqref{est:lnk2} and get
\[
\exp\left\{-\sum_{i=j}^{n} \frac 1{\left(i+e_{[k]}^1\right)\ln \left(i+e_{[k]}^1\right)\dots \ln_k \left(i+e_{[k]}^1\right)}\right\}\le \frac{ \ln_{k} (j+e_{[k]}^1)}{ \ln_{k} (n+1+e_{[k]}^1)},
\]
and 
\[
\exp\left\{-\sum_{i=0}^{n} \frac 1{\left(i+e_{[k]}^1\right)\ln \left(i+e_{[k]}^1\right)\dots \ln_k \left(i+e_{[k]}^1\right)}\right\}\le \frac{1}{ \ln_{k} (n+1+e_{[k]}^1)}.
\]
 Applying all the above we arrive at
\begin{equation}
\label{est:n11}
\begin{split}
|x_{n+1}|^{\frac 1{3^{n+1}}}&\le  \frac{|x_0|}{ \ln_{k} (n+1+e_{[k]}^1)}+  \frac{\sum_{j=1}^{n+1}\ln_{k} (j+e_{[k]}^1)|u_j|^{\frac 1{3^j}}}{ \ln_{k} (n+2+e_{[k]}^1)}
\\
&\le  \frac{|x_0| \ln_{k}}{ \ln_{k} (n+2+e_{[k]}^1)}+  \frac{\sum_{j=1}^{n+1} \frac \beta{(j+e_{[k]}^1)\ln (j+e_{[k]}^1)\dots \ln_{k-1} (j+e_{[k]}^1)}}{ \ln_{k} (n+2+e_{[k]}^1)}
\\
&= \frac{|x_0|}{ \ln_{k} (n+1+e_{[k]}^1)}+ \beta\left(\ln_{k} (n+2+e_{[k]}^1)\right)^{-1}\ln_{k} (n+2+e_{[k]}^1)\\&= \frac{|x_0|}{ \ln_{k} (n+1+e_{[k]}^1)}+ \beta.
\end{split}
\end{equation}
So for each $\beta\in (0, 1)$ we can find $N_1$ such that, for $n\ge N_1$, 
\[
\frac{|x_0|}{ \ln_{k} (n+1+e_{[k]}^1)}+ \beta<1, 
\]
which implies $|x_{N_1+1}|<1$. 
Assume now that $N_2>2$ is such that, for $n\ge N_2$, we have 
\[
\left(n+e_{[k]}^1\right)\ln \left(n+e_{[k]}^1\right)\dots \ln_k \left(n+e_{[k]}^1\right)\le 3^{\frac n2}.
\]
Then,  for $n\ge N_2$, 
\begin{equation}
\label{est:hnN2}
h_n\le e^{-\frac{3^{n+1}}{\left(n+e_{[k]}^1\right)\ln \left(n+e_{[k]}^1\right)\dots \ln_k \left(n+e_{[k]}^1\right)}}<e^{-3^{\frac n2+1}}\le e^{-3^{2}}=e^{-9}.
\end{equation}
Without loss of generality we can assume that $N_1\ge N_2$.
We have
\[
0<1-h_{N_1+1}x_{N_1+1}^2<1,\quad |x_{N_1+2}|< |x_{N_1+1}|+|u_{N_1+2}|.
\]
Also
\[
x_{N_1+2}^2< 2x_{N_1+1}^2+2u_{N_1+2}^2, \quad \mbox{and}\quad |u_n|<1, \quad \forall n\in \mathbb N,
\]
so
\begin{equation}
\begin{split}
\label{est:nx2}
&h_{N_1+2}x_{N_1+2}^2< 2h_{N_1+2}  \left[x_{N_1+1}^2+u_{N_1+2}^2\right]=
2e^{-9}\left[x_{N_1+1}^2+1\right]\\&
= 4e^{-9}\approx 0.00049<1.
\end{split}
\end{equation}
Based on that we get 
\[
|x_{N_1+3}|< |x_{N_1+2}|+|u_{N_1+3}|< |x_{N_1+1}|+|u_{N_1+2}|+|u_{N_1+3}|.
\]
Applying induction, assume that, for some $k\in \mathbb N$, 
\begin{equation}
\label {ineq:ind11}
|x_{N_1+2+k}|\le  |x_{N_1+1}|+\sum_{i=1}^{k}|u_{N_1+2+i}| \quad \mbox{and} \quad x^2_{N_1+2+k}h_{N_1+2+k}<1,
\end{equation}
and prove that relations in \eqref{ineq:ind11} hold for $k+1$.  In order to do so we first get the estimate of $\sum_{i=1}^{k}|u_{N_1+2+i}|$. For all $n\in \mathbb N$, we have
\[
|u_{n}|\le  \left( \frac \beta{\left(n+e_{[k]}^1\right)\ln \left(n+e_{[k]}^1\right)\dots \ln_k \left(n+e_{[k]}^1\right)}\right)^{3^{n}}<\left(\frac \beta {n+e_{[k]}^1}\right)^{3^{n}}.
\]
Then, for $n\ge N_1+2\ge 4$,
\[
|u_{n}|\le  \left(\frac \beta n\right)^{3^{n}}<\left(\frac \beta n\right)^{n}\le \left(\frac \beta 4\right)^{n}
\]
and 
\begin{equation}
\label {est:sumun}
\sum_{i=1}^{k}|u_{N_1+2+i}|<\sum_{i=1}^{\infty}|u_{N_1+2+i}|\le \sum_{n=4}^{\infty} \left(\frac \beta 4\right)^{n}=\frac{\left(\frac \beta 4\right)^{4}}{1-\frac \beta 4}<\frac{4\left(\frac 14\right)^{4}}{4-\beta}<\frac1{4^3\times 3}<1.
\end{equation}
Now, 
\begin{equation*}
|x_{N_1+2+k+1}|\le  |x_{N_1+2+k}|+ |u_{N_1+2+k+1}|\le  |x_{N_1+1}|+\sum_{i=1}^{k+1}|u_{N_1+2+i}|,
\end{equation*}
proving the first part of \eqref{ineq:ind11} for each $k\in \mathbb N$, and 
\begin{equation*}
\begin{split}
h_{N_1+2+k+1}x^2_{N_1+2+k+1}&\le 2h_{N_1+2+k+1} |x_{N_1+1}|^2+2h_{N_1+2+k+1}\left(\sum_{i=1}^{k+1}|u_{N_1+2+i}|\right)^2\\&
\le 2 e^{-9}\left[1+1\right]\le 4e^{-9}<1,
\end{split}
\end{equation*}
proving the second part of \eqref{ineq:ind11} for each $k\in \mathbb N$. This completes the proof of Part (i). 

 From \eqref{ineq:ind11} and \eqref{est:sumun} we have 
\begin{eqnarray*}
&&|x_{N_1+2+k}|< |x_{N_1+1}|+1,
\end{eqnarray*}
for each $k\in \mathbb N$, which completes the proof of Part (ii).
\end{proof}

\subsubsection{Case 2}
\label{subsub:e3n}
Assume that, for some $\beta\in (0, 1)$,
\begin{equation}
\label{def:hu}
h_n\le e^{-3^{n+1}}, \quad |u_n|\le \left[ \frac {\beta(e-1)}{e}\right]^{3^{n}}.
\end{equation}
\begin{lemma}
\label{lem:N11}
The statement of Lemma \ref{lem:N1} holds if, instead of conditions \eqref{def:h1}--\eqref{def:u1}, we assume that condition \eqref{def:hu} holds.
\end{lemma}

\begin{proof}
The proof is analogous to the proof of Lemma \ref{lem:N1}.
Instead of \eqref{est:n11} we obtain
\begin{equation}
\label{est:n2}
\begin{split}
|x_{n+1}|^{\frac 1{3^{n+1}}}&\le e^{-(n+1)}|x_0|+ \frac {\beta(e-1)}{e} [e^{-n}+e^{-n+1}+\dots+ e^{-1}+1]
\\&
=e^{-(n+1)}|x_0|+  \frac {\beta(e-1)}{e}\frac{1-e^{-n-1}}{1-e^{-1}}\le e^{-(n+1)}|x_0|+ \beta.
\end{split}
 \end{equation}
Taking $ N_1\ge \ln |x_0|-\ln [1-\beta]$ we get $|x_{n+1}|<1$ for $n\ge N_1$. Instead of \eqref{est:nx2} we have 
\begin{equation*}
\begin{split}
&h_{N_1+2}x_{N_1+2}^2< 2h_{N_1+2}  \left[x_{N_1+1}^2+u_{N_1+2}^2\right]=
2e^{-3^{N_1+3}}\left[x_{N_1+1}^2+\left[ \frac {\beta(e-1)}{e}\right]^{2\cdot 3^{N_1+2}}\right]\\&
\le 2e^{-3^{N_1+3}}\left[1+1\right]\le 4e^{-3^{N_1+3}}<4e^{-3^4}<1,
\end{split}
\end{equation*}
and instead of \eqref{est:sumun} we have 
\begin{equation*}
\begin{split}
&\sum_{i=1}^{k}|u_{N_1+2+i}|= \sum_{i=1}^{k} \left[ \frac {\beta(e-1)}{e}\right] ^{3^{N_1+2+i} }\le \sum_{j=4}^{k} \left[ \frac {\beta(e-1)}{e}\right] ^{3^j}\\&< \sum_{j=4}^{k} \left[ \frac {\beta(e-1)}{e}\right] ^{j}
=\left[ \frac {\beta(e-1)}{e}\right] ^{3^4}\frac 1{1- \frac {\beta(e-1)}{e}}\\&<\left[ \frac {\beta(e-1)}{e}\right] ^{3^4} e<1.
\end{split}
\end{equation*}
The last inequality holds true since, in particular, 
\[
\left[ \frac {(e-1)}{e}\right] ^{3^4}\approx (0.6321)^81<0.3678\approx e^{-1}
\]
The rest of the proof is similar to the proof of Lemma \ref{lem:N1}. 
 \end{proof}
\subsection{Convergence of $(x_n)_{n\in\mathbb{N}}$ to a finite limit.}

\begin{lemma}
\label{lem:N1conv}
Let $(x_n)_{n\in\mathbb{N}}$ be a solution to equation \eqref{eq:main3} and let $(h_n)_{n\in\mathbb{N}}$ and $(u_n)_{n\in\mathbb{N}}$ satisfy either conditions \eqref{def:h1}-\eqref{def:u1} or condition \eqref{def:hu}. Then the sequence $(x_{k})_{k\in \mathbb N}$ converges to a finite limit as $k\to\infty$.
\end{lemma}
\begin{proof}
It is sufficient to consider only the terms $\{ x_{N_1+2+k} \}_{k\in \mathbb N}$. Since the sequence $\{ x_{N_1+2+k} \}_{k\in \mathbb N}$ is bounded, it has a convergent subsequence $\{ x_{N_1+2+k_l} \}_{l\in \mathbb N}$,
\[
\lim_{l\to \infty}x_{N_1+2+k_l}=L.
\]
We now show that 
 \[
\lim_{m\to \infty}x_{N_1+2+m}=L
\]
follows. For each $m\in\mathbb N$ denote $l_m\in \mathbb N$
\[
l_m=\sup\{l: N_2+2+k_l\le m \}.
\]
Then 
\[
N_2+2+k_{l_m}\le m\le N_2+2+k_{l_{m}+1}
\]
and
\begin{equation}
\label{est:up}
|x_{N_1+2+m}|\le |x_{N_1+2+m-1}| +|u_{N_1+2+m}|\le |x_{N_1+2+k_{l_m}}| +\sum_{i=k_{l_m}}^{m}|u_{N_1+2+i}|,
\end{equation}
\begin{equation}
\label{est:down}
|x_{N_1+2+k_{l_{m}+1}}|\le |x_{N_1+2+m}| +\sum^{k_{l_m+1}}_{i=m}|u_{N_1+2+i}|
\end{equation}
Passing to the limit in \eqref{est:up} and  \eqref{est:down} we obtain, respectively,
\[
\limsup_{m\to \infty}x_{N_1+2+m}\le L, \quad \mbox{and} \quad L\le \liminf_{m\to \infty}x_{N_1+2+m}.
\]
This implies that $\lim_{m\to \infty}x_{N_1+2+m}$ exists and equal to $L$.

\end{proof}

When condition \eqref{def:hu} holds it is possible that solutions of \eqref{eq:main3} converge to a nonzero limit. Example \ref{ex:0non0} below demonstrates  that $\lim_{n\to \infty} x_n$ can be either zero or nonzero.

\begin{example}
\label{ex:0non0}
We show that the limit of solutions of \eqref{eq:main3} can be positive, zero, or negative. For all three cases below, choose $h_n=e^{-3^{n+1}}$.
\begin{enumerate}
\item[(i)] {\bf Zero limit ($L=0$).} Set 
\[
u_1=-e^{-3} \approx -0.0498, \quad u_n=0\quad\text{for all}\quad n\geq 2.
\]
Then \eqref{def:hu} is satisfied for $\beta\in(1/(e-1),1)$. The continuous function
\[
f(x)=x-e^{-3}x^3.
\]
takes its maximum $f_m=\frac 2{3\sqrt{3e^{-3}}}\approx 1.724>0.0498\approx -u_1$ at the point $x_m=\frac 1{\sqrt{3e^{-3}}}\approx 2.586$, and $f(0)=0$. So the equation 
\[
x-e^{-3}x^3=e^{-1},
\]
has a solution $x_0$ on the interval $\left(0,  \frac 1{\sqrt{3e^{-3}}}\right) \approx (0, 2.586)$. Consider now the equation \eqref{eq:main3} with this specific initial value.  We get $x_1=0$ and since all $u_n=0$ for $n\geq 2$, we have $x_n=0$ for $n\geq 2$. Therefore $\lim_{n\to\infty}x_n=0$.

\bigskip
\item[(ii)]  {\bf Positive limit ($L>0$).} Set
\[
u_1=e^{-3} \approx 0.0498, \quad u_n> 0,\quad\text{for all}\quad n\geq 2,
\]
so that \eqref{def:hu} is satisfied. Suppose also that $x_0>0$ is chosen as in case (i). Then,
\[
x_{1}=2u_1+ \underbrace{x_0(1-h_0x_0^2)-u_1}_{=0}=2u_1=2e^{-3}>0.
\] 
Moreover, note that $h_1x_1^2=2e^{-12}<1/2$. We can also write
\[
x_{n+1}\geq x_n(1-h_nx_n^2)\geq x_1\prod_{i=1}^{n}(1-h_ix_i^2).
\]
The same approach as in Lemma \ref{lem:finite} with $N=1$ gives that  $\lim_{n\to \infty} x_n>0$.

\item[(iii)] {\bf Negative limit ($L<0$).} Set
\[
u_1=-2e^{-3} \approx -0.0996, \quad u_n< 0,\quad\text{for all}\quad n\geq 2,
\]
so that \eqref{def:hu} is satisfied, and choose $x_0>0$ as in Cases (i) and (ii). Then
\[
x_1=\underbrace{x_0(1-h_0x_0^2)+\frac{u_1}{2}}_{=0}+\frac{u_1}{2}<0.
\]
Again, we see that $h_1x_1^2=e^{-18}<1/2$, and we can write for all $n\geq 1$
\[
x_{n+1}\le x_n(1-h_nx_n^2)\le x_1\prod_{i=1}^n(1-h_ix_i^2).
\]
The same approach as in Lemma \ref{lem:finite} with $N=1$ gives that  $\lim_{n\to \infty} x_n<0$.
\ee
\end{example}

\subsection{Modified process with a stopped timestep sequence $(h_n)_{n\in\mathbb{N}}$}
Based on Example \ref{ex:0non0} and Lemma \ref{lem:finite} we cannot expect that, in general, the finite limit $L$ will be zero. In order to obtain a sequence that converges to zero we modify the timestep sequence $(h_n)_{n\in\mathbb{N}}$ further by stopping it (preventing terms from varying further) after $N_3$ steps:
\begin{equation}
\label {def:hath}
\hat h_n=
\left \{
  \begin{array}{cc}
  h_n,  &  n<N_3, \\
  h_{N}, & n\ge N_3,
    \end{array}
\right.
\end{equation}
where $N_3$ is such that 
\begin{equation}
\label {ineq:N3}
|x_{N_3}|\le 1.
\end{equation} 
Note that under the conditions of Lemmas \ref{lem:N1} and \ref{lem:N11} we would have $N_3=N_1$. Note that $N_3$ is not necessarily the first moment where \eqref{ineq:N3} holds, and that \eqref{ineq:N3} implies $x_{N_3}^2h_{N_3}<1$, but the converse does not necessarily hold. 

Consider
\begin{equation}
\label{eq:main3mod}
x_{n+1}=x_n(1-\hat h_nx_n^2)+u_{n+1}, \quad x_0\in \mathbb R.
\end{equation}

\begin{lemma}
\label{lem:stop}
Let $(h_n)_{n\in\mathbb{N}}$ and $(u_n)_{n\in\mathbb{N}}$ satisfy either conditions \eqref{def:h1}-\eqref{def:u1} or condition \eqref{def:hu}. Let  $(x_n)_{n\in\mathbb{N}}$ be a solution to equation \eqref{eq:main3mod} with   $(\hat h_n)_{n\in\mathbb{N}}$ defined by \eqref {def:hath}. Then  $\lim_{n\to \infty}x_n=0$  for any initial value $x_0\in \mathbb R$.
\end{lemma}
\begin{proof}
Choose $N_1$ defined as in Lemmata \ref{lem:N1} or \ref{lem:N11} and set $N_3=N_1$.  To prove that 
\[
x_n^2\hat h_n< 1,\quad \mbox{for all} \quad n>N_3,
\]
we follow the approach taken in the proofs of Lemma \ref{lem:N1}, Part (i), and Lemma \ref{lem:N11}, Part (i). 

Let assume first  that  conditions \eqref{def:h1}--\eqref{def:u1} hold, so we use $N_1$ from Lemma \ref{lem:N1}. We have $N_3=N_1>2$, $|x_{N_3}|<1$, 
$\hat h_{N_3+1}<e^{-3^{\frac{ N_3}2+1}}$,
\[
|x_{N_3+1}|\le |x_{N_3}|+|u_{N_3+1}|
\]
and
\begin{equation*}
\begin{split}
&\hat h_{N_3+1}x_{N_3+1}^2< 2\hat h_{N_3+1}  \left[x_{N_3}^2+u_{N_3+1}^2\right]=
2e^{-3^{\frac{N_3}2+1}}\left[x_{N_3}^2+\left(\frac \beta 4\right)^{N_3}\right]\\&
\le 2e^{-3^{2}}\left[1+1\right]=4e^{-3^2}<1.
\end{split}
\end{equation*}
This gives us
\[
|x_{N+2}|\le |x_{N+1}|+|u_{N+2}|,
\]
which, as above, leads to 
\begin{eqnarray*}
\hat h_{N+2}x_{N+2}^2&<& 2\hat h_{N_3}  \left[x_{N_3}^2+u_{N_3+1}^2\right]\le 
2e^{-3^{\frac{N_3+1}2+1}}\left[1+\left(\frac \beta 4\right)^{N_3+1}\right]\\
&<&4e^{-3^2}<1.
\end{eqnarray*}
Now we complete the proof by induction and arrive at
\begin{equation}
\label{eq:lim}
|x_{N+k}|\le |x_N| +\sum_{i=1}^k|u_{N+i}|,
\end{equation}
which implies the boundedness of the sequence $(x_n)_{n\in \mathbb N}$.  Note that Lemma \ref{lem:N1conv} also holds when, instead of $(h_n)_{n\in\mathbb{N}}$  we have a stopped sequence $(\hat h_n)_{n\in\mathbb{N}}$, since its proof uses only \eqref{eq:lim} and convergence of the series $\sum_{i=1}^\infty u_i$.  So we conclude that $\lim_{n\to \infty}x_n=L$.
Passing to the limit in equation \eqref{eq:main3mod} we obtain the equality
\[
L=L(1-\hat h_N L),
\]
which holds only for $L=0$.

If  condition \eqref{def:hu} hold, we use $N_1$ from Lemma \ref{lem:N11}. The proof of this case is similar to that of the first, except that $\hat h_n\le 3^{N_3+1}$.
\end{proof}

\begin{remark}
\label{rem:quiklier}
Convergence of the solutions of equation \eqref{eq:main3mod} with stopped time-step sequence $(\hat h_n)_{n\in\mathbb{N}}$ may be slow, either if $h_{N_3}$ is very small, or if $N_3$ is large. Alternative strategies for stopping the sequence $(\hat h_n)_{n\in\mathbb{N}}$ are as follows: 
\be

\item [(i)] Define
\begin{equation}
\label{eq:N4}
N_4=\inf\{n\in \mathbb N: x_n^2h_n<1 \}, 
\end{equation}
and assume that $x_{N_4}\neq 0$.  Define
\[
\quad \hat h_n=
\left \{
  \begin{array}{cc}
  h_n,  &  n<N_4, \\
  \frac 1{x^2_{N_4}}, & n\ge N_4.
    \end{array}
\right.
\]
Then, $
|x_{N_4+1}|=|u_{N_4+1}|<1,$
and the conditions of Lemma \ref{lem:stop} hold. If $x_{N_4}= 0$, we also have  $
|x_{N_4+1}|=|u_{N_4+1}|<1.$\\

\item [(ii)] Assume that $|u_{n+1}|\le h_n$ for all $n\in \mathbb N$. Define again $N_4$ by \eqref{eq:N4}.   If $|x_{N_4}|\le 1$ the conditions of Lemma \ref{lem:stop} hold. If  $|x_{N_4}|>1$, we have
\[
|x^3_{N_4}|>1\ge \frac{|u_{N_4+1}|}{h_{N_4}}, \,\, \mbox{or} \,\, |x^3_{N_4}|h_{N_4}\ge |u_{N_4+1}|.
\]
Then,
\begin{multline*}
|x_{N_4+1}|\le |x_{N_4}|(1-x^2_{N_4}h_{N_4})+|u_{N_4+1}|\\= |x_{N_4}|-|x^3_{N_4}|h_{N_4}+|u_{N_4+1}|\le  |x_{N_4}|.
\end{multline*}
So
\[
x^2_{N_4+1}\hat h_{N_4+1}\le x^2_{N_4}h_{N_4}\le 1.
\] 
By induction it can be shown that $x^2_{N_4+k}\hat h_{N_4+k}\le 1$ for all $k\in \mathbb N$. Now, applying the same reasoning as before we can prove that $(|x_{N_4+k}|)_{k\in \mathbb N}$ is uniformly bounded and converges to zero. 
\ee
\end{remark}

\begin{lemma}
\label{lem: logcond}
Let $(h_n)_{n\in\mathbb{N}}$ and $(u_n)_{n\in\mathbb{N}}$ satisfy either conditions \eqref{def:h1}-\eqref{def:u1} or condition \eqref{def:hu} with $\beta<\frac 1{e-1}$, in all cases with equality instead of inequality in the conditions placed upon each $h_n$. Let  $(x_n)_{n\in\mathbb{N}}$ be a solution to equation \eqref{eq:main3mod} with initial value $x_0\in \mathbb R$ and  $(\hat h_n)_{n\in\mathbb{N}}$ defined by \eqref {def:hath} and \eqref{eq:N4}. Then  $\lim_{n\to \infty}x_n=0$.
\end{lemma}

\begin{proof}
By Lemmas  \ref{lem:N1}, \ref{lem:N11} and Remark \eqref{rem:quiklier}, Part (ii), it is sufficient to show that
 $|u_{n+1}|\le h_n$. Denote 
 \[
Q(n):=\ln \beta-\sum_{i=1}^{k+1} \ln_{i}\left(n+e^1_{[k]}  \right)+\left(\prod_{i=0}^{k} \ln_{i}\left(n+e^1_{[k]}  \right)\right)^{-1}
\]
Note that, for $n\ge 1$, 
\begin{equation*}
 \begin{split}
 Q(n)&<\ln \beta- \ln\left(n+e^1_{[k]}  \right)+\frac 1{ \left(n+e^1_{[k]}\right)}\\
 &\le  \ln \beta -\ln 2+\frac 12\\
 &\approx \ln \beta-0.1931<0.
 \end{split}
  \end{equation*} 
 When conditions \eqref{def:h1}--\eqref{def:u1} hold we have, for $n\ge 1$,
   \begin{equation*}
 \begin{split}
 \frac{|u_{n+1}|}{h_n }\le \exp\left\{3^{n+1} Q(n) \right\}\le 1.
  \end{split}
  \end{equation*} 
 When  condition \eqref{def:hu} holds with $\beta(e-1)\le 1$, we have, for $n\ge 1$, 
 \[
 \frac{|u_{n+1}|}{h_n }\le (\beta(e-1))^{3^{n+1}}\le 1.
  \]
  \end{proof} 

\section{The stochastically perturbed cubic difference equation}\label{sec:randompert}
In this section we consider a stochastic difference equation
\begin{equation}
\label{eq:stoch}
x_{n+1}=x_n(1-h_nx_n^2)+\rho_{n+1}\xi_{n+1}, \quad n\in \mathbb N,\quad  x_0\in \mathbb R,
\end{equation}
where $(\xi_n)_{n\in \mathbb N}$ is a sequence of independent identically distributed random variables.
We discuss  only two cases: $|\xi_n|\le 1$ and $\xi_n\sim \mathcal N(0, 1)$.
 Denoting 
\[
u_n:=\rho_{n}\xi_{n}
\]
we can apply the results of Section \ref{sec:nohomdet} pathwise to solutions of \eqref {eq:stoch} for almost all $\omega\in \Omega$.

We also consider a stochastically perturbed equation with stopped timestep sequence $(\hat h_n)_{n\in\mathbb{N}}$
\begin{equation}
\label{eq:stochhat}
x_{n+1}=x_n(1-\hat h_nx_n^2)+\rho_{n+1}\xi_{n+1}, \quad n\in \mathbb N,\quad  x_0\in \mathbb R,
\end{equation}
where $\hat h_n$ is defined by \eqref{def:hath}  with $N_3$ selected as equal to $N_1$ from Lemmas \ref{lem:N1}, \ref{lem:N11} or as equal to $N_4$ from Remark \ref{rem:quiklier}. Note that since  solutions of \eqref{eq:stoch} are stochastic processes, $N_1$ and $N_4$ are  a.s. finite $\mathbb{N}$-valued random variables, which we therefore denote by $\mathcal N_1$ and $\mathcal N_4$, respectively.
   
\subsection{Case 1: bounded noise ($|\xi_n|\le 1$)} 
\label{subsec:bound}

In this case, for all $\omega\in \Omega$ and all $n\in \mathbb N$, we have 
\[
|u_n|=|\rho_n\xi_n|\le |\rho_n|.
\]
for all $\omega\in\Omega$. So we may apply the results of Section \ref{sec:nohomdet} to each trajectory, arriving at 
 \begin{theorem}
 \label{thm:stochunif1}
 Let $(h_n)_{n\in\mathbb{N}}$ and $(\rho_n)_{n\in\mathbb{N}}$ satisfy either conditions \eqref{def:h1}-\eqref{def:u1} or condition \eqref{def:hu} ($\rho_n$ satisfying the constraint for $u_n$). 
  Let $(\xi_n)_{n\in \mathbb N}$ be  a sequence of random variables s.t. $|\xi_n|\le 1$ for all $n\in \mathbb N$. 
 Let $(x_n)_{n\in\mathbb{N}}$ be a solution to \eqref{eq:stoch}, $(\hat h_n)_{n\in\mathbb{N}}$ defined as in \eqref{def:hath}, and $(\hat x_n)_{n\in\mathbb{N}}$ a solution to \eqref{eq:stochhat}. Then, a.s., 
 \be 
 \item[(i)] $\lim_{n\to \infty} x_n=L$, where $L$ is an a.s. finite random variable;
 \item[(ii)]  $\lim_{n\to \infty} \hat x_n=0$.
 \ee
  \end{theorem}
  
\subsection{Case 2: unbounded noise ($\xi_n\sim\mathcal N(0, 1)$).}

\begin{theorem}
 \label{thm:stochunif2}
 Let $(h_n)_{n\in\mathbb{N}}$ and $(\rho_n)_{n\in\mathbb{N}}$ satisfy either conditions \eqref{def:h1}-\eqref{def:u1} or condition \eqref{def:hu} ($\rho_n$ satisfying the constraint for $u_n$).   Let $(\xi_n)_{n\in \mathbb N}$ be  a sequence of mutually independent $\mathcal N(0, 1)$ random variables.  Let $(x_n)_{n\in\mathbb{N}}$ be a solution to \eqref{eq:stoch}, $(\bar h_n)_{n\in\mathbb{N}}$ as defined in \eqref{def:hath}, and $(\hat x_n)_{n\in\mathbb{N}}$ a solution to \eqref{eq:stochhat}. Then, a.s., 
 \be 
 \item[(i)] $\lim_{n\to \infty} x_n=L$, where $L$ is an a.s. finite random variable;
 \item[(ii)]  $\lim_{n\to \infty} \hat x_n=0$.
 \ee
  \end{theorem}

\begin{proof}
If \eqref{def:hu} holds for  $\beta\in (0, 1)$, then for some $\beta_1\in (\beta, 1)$ we have
\[
\left[\frac {\beta(e-1)}{e}\right]^{3^n}=\left[\frac {\beta_1(e-1)}{e}\right]^{3^n}\times \left[\frac {\beta}{\beta_1}\right]^{3^n},
\]
and, for each $\varsigma>0$,
\[
\lim_{n\to 0}\left[\frac {\beta}{\beta_1}\right]^{3^n}\ln^{\frac 12+\zeta} n = 0,
 \]
 Applying Lemma \ref{lem:norm} we conclude that there exists $\mathcal N$ such that for all $n\ge \mathcal N$,
\[
\left|\frac 1{(\ln n)^{1/2+\varsigma}} \xi_n\right|<1.
\]
Then, for all $n\ge \mathcal N$,
\[
|u_{n+1}|=\left|\left[\frac {\beta_1(e-1)}{e}\right]^{3^n}\times \left[\frac {\beta}{\beta_1}\right]^{3^n}\xi_n\right|\le \left[\frac {\beta_1(e-1)}{e}\right]^{3^n}.
\]
If \eqref{def:u1} hold  holds for  $\beta\in (0, 1)$, then for some $\beta_1\in (\beta, 1)$ we use the estimate
\[
 |u_{n+1}|\le  \left( \frac{\beta_1}{\left(n+e_{[k]}^1\right)\ln \left(n+e_{[k]}^1\right)\dots \ln_k \left(n+e_{[k]}^1\right)}\right)^{3^{n} } \left[\frac {\beta}{\beta_1}\right]^{3^n}|\xi_{n+1}|,
 \] 
and apply the same reasoning  as above. 

Define for a.a. $\omega\in \Omega$ 
\[
y_m:=x_{m+\mathcal N(\omega)}, \quad u_{m+1}:=\rho_{m+\mathcal N(\omega)}\xi_{m+\mathcal N(\omega)} (\omega), \quad {\rm h}_m:= h_{m+\mathcal N(\omega)},
\]
and consider the deterministic stochastic equation
\begin{equation}
\label{eq:stochN}
y_{m+1}=y_m(1-{\rm h}_my_m^2)+u_{m+1}, \quad m\in \mathbb N, \quad y_0=x_{\mathcal N(\omega)}.
\end{equation}
Equation \eqref{eq:stochN} satisfies the conditions of either Lemma \ref{lem:N1} or Lemma  \ref{lem:N11}. So there exists 
$N_1$ (which depends on $\omega$)
such that $ h_{N_1}x_{N_1}^2<1$. The remainder of the proof follows by the same argument as that in Section \ref{sec:nohomdet}.
\end{proof}

\section{Illustrative numerical examples}
\label{sec:exsim}
In this section we illustrate the asymptotic behaviour of solutions of the unperturbed equation \eqref{eq:main3hom} with summable and non-summable timestep sequences, as described in Lemmas \ref{lem:finite} \& \ref{lem:infinite}, and the stochastically perturbed equation \eqref{eq:stoch} with unbounded Gaussian noise as described in Theorem \ref{thm:stochunif2}. 

Figure \ref{fig:detplots}, parts (a) and (b) provide three solutions of the unperturbed deterministic equation \eqref{eq:main3hom} corresponding to the initial values $x_0=1.1,0.5,-1.1$, with timestep sequence $h_n=1/n^{10}$, so that $\sum_{i=1}^{\infty}h_i<\infty$. We observe that all three solutions appear to converge to different finite limits, as predicted by Lemma \ref{lem:finite}. 

Parts (c) and (d) provide three solutions of \eqref{eq:main3hom} with the same initial values and with timestep sequence $h_n=1/n^{0.1}$, so that $\sum_{i=1}^{\infty}h_i=\infty$. Note that we have selected values of $x_0$ that are sufficiently small for \eqref{ineq:<} to hold with this choice of $h_n$, hence avoiding the possibility of blow-up. All three solutions appear to converge to a zero limit, as predicted by Lemma \ref{lem:infinite}. 

Figure \ref{fig:stochplots}, part (a) and (b) provide three solution trajectories of the stochastic equation \eqref{eq:stoch} each corresponding to initial value given by $x_0=2.5,0.5,-2.5$ with timestep sequence $h_n=e^{-\frac{3^{n+1}}{n+e}}$, satisfying \eqref{def:h1} for $k=1$, $(\xi_n)_{n\in\mathbb{N}}$ a sequence of i.i.d. $N(0,1)$ random variables, and
\begin{equation}\label{eq:comprho}
\rho_n=\left(\frac{\beta}{n+e}\right)^{3^n},
\end{equation}
with $\beta=0.5$ satisfying \eqref{def:u1} with $k=1$. We observe that all three solutions approach different nonzero limits, as predicted by Theorem \ref{thm:stochunif2}.

Parts (c) and (d) repeat the computation, but with the timestep sequence stopped so that its values become fixed when $h_nx_n^2<1$ is satisfied for the first time. Solutions demonstrate behaviour consistent with asymptotic convergence to zero, also as predicted by Theorem \ref{thm:stochunif2}.

Note that $\beta\in(0,1)$ in Condition \eqref{def:u1}, but that in Figure \ref{fig:stochplots} the effect of the stochastic perturbation decays too rapidly for differences between trajectories to be visible. Therefore in each part of Figure 3 we choose larger values of $\beta$ and generate fifteen trajectories of \eqref{eq:stoch} with $(\xi_n)_{n\in\mathbb{N}}$ a sequence of i.i.d. $\mathcal{N}(0,1)$ random variables, timestep sequence $h_n=e^{-\frac{3^{n+1}}{n+e}}$ stopped when $h_nx_n^2<1$ is satisfied for the first time, $x_0=2.5$, and each $\rho_n$ chosen to satisfy \eqref{eq:comprho}. Parts (a) and (b) show that, when $\beta=3/2$, trajectories appear to converge to zero. However, Parts (c) and (d) show that, when $\beta=3$ and $\beta=5$ respectively, trajectories may converge to a random limit that is not necessarily zero a.s.

\begin{figure}
\begin{center}
$\begin{array}{@{\hspace{-0.3in}}c@{\hspace{-0.3in}}c}
\mbox{\bf\small Short-term} & \mbox{\bf\small Long-term}\\
\scalebox{0.33}{\includegraphics{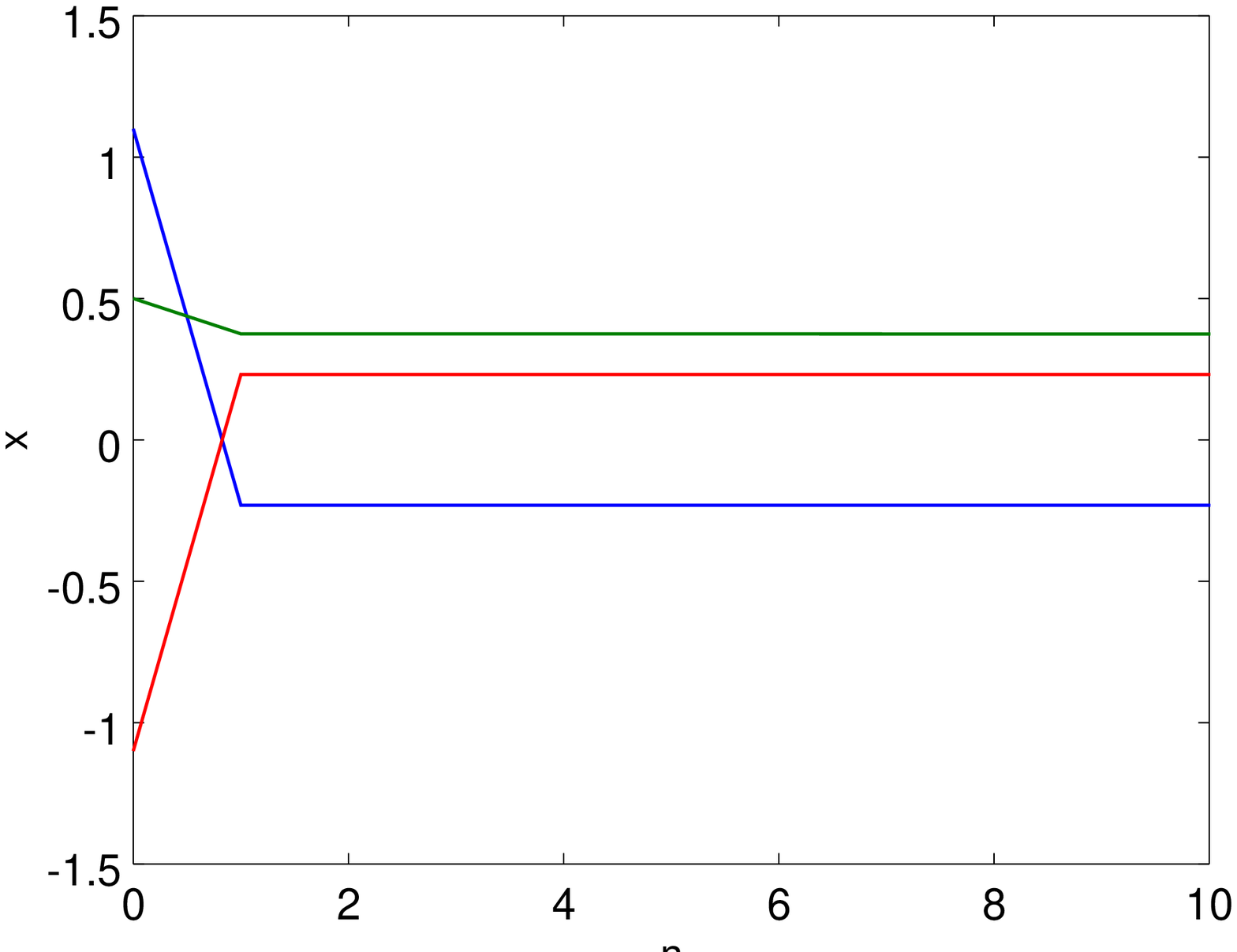}} & \scalebox{0.33}{\includegraphics{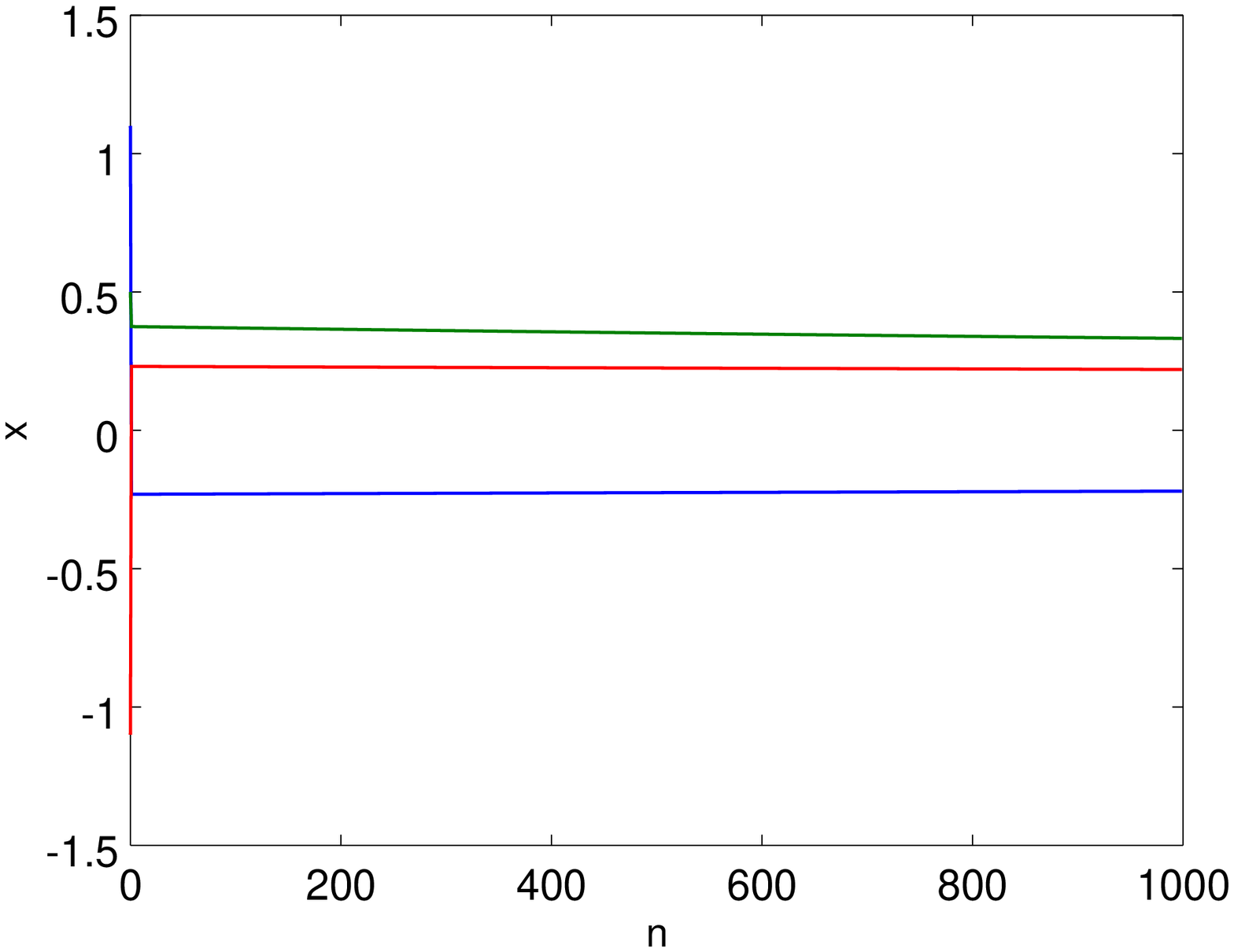}}\\ \mbox{\bf\small (a)} & \mbox{\bf\small (b)}\\
\scalebox{0.33}{\includegraphics{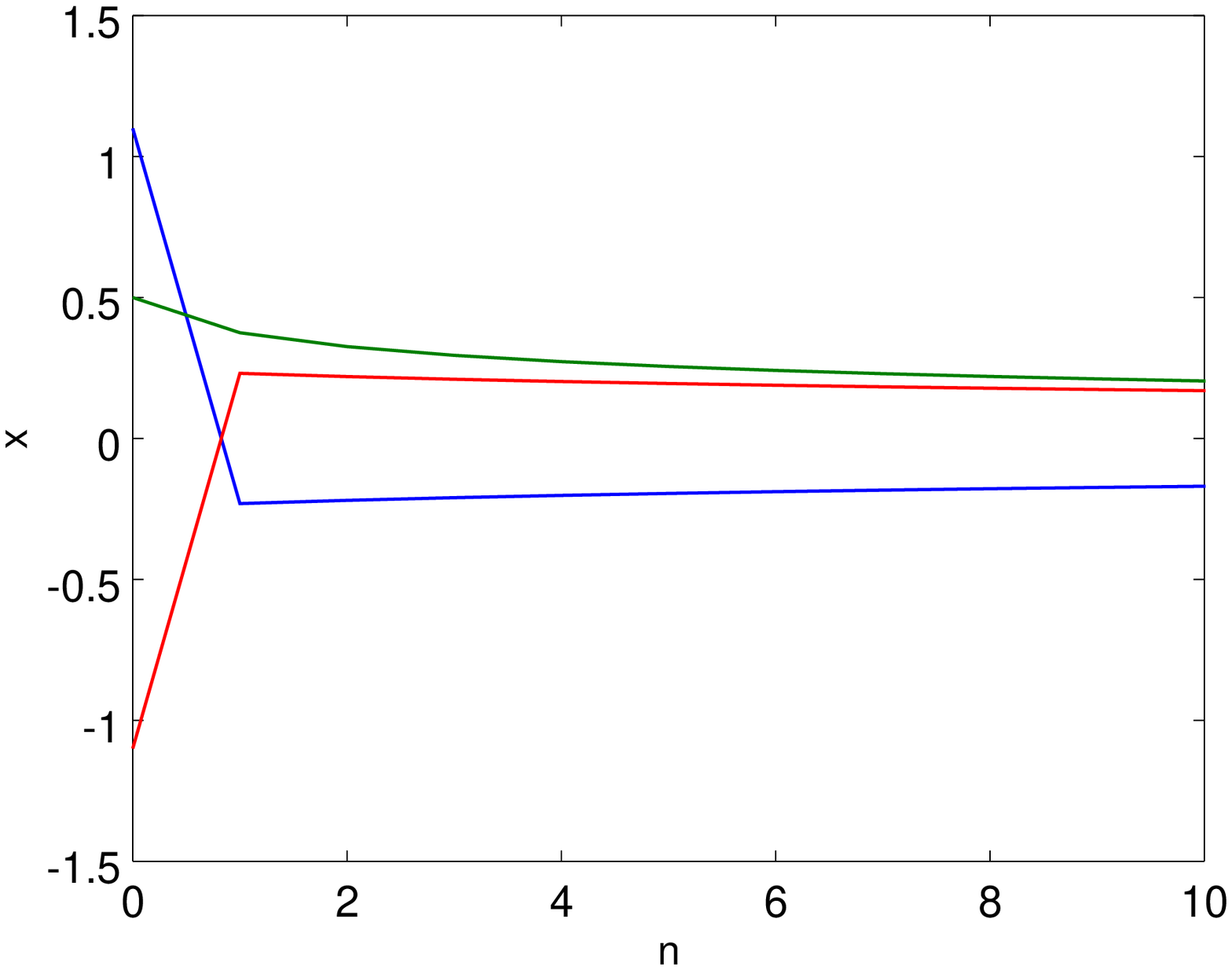}} & \scalebox{0.33}{\includegraphics{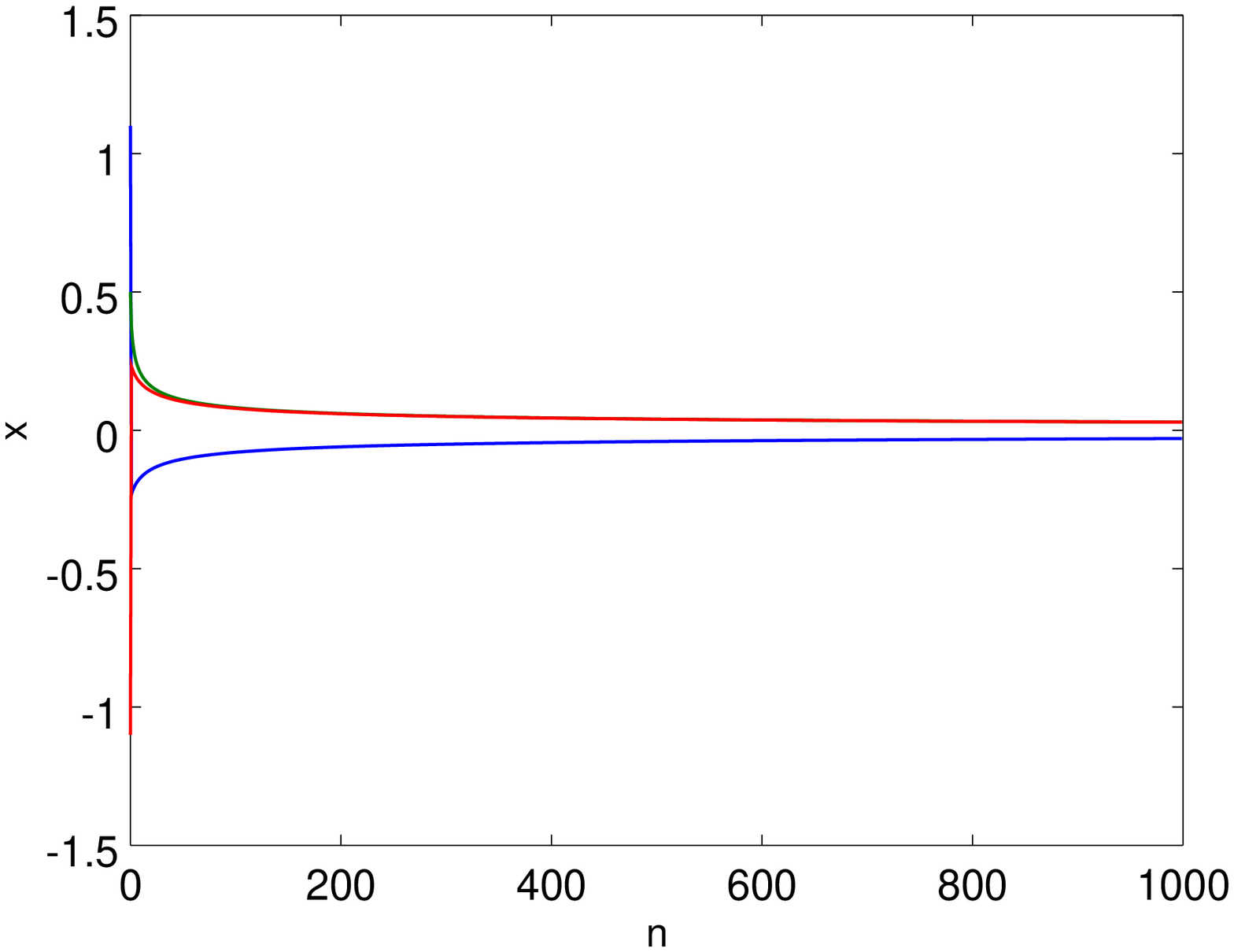}}\\ \mbox{\bf\small (c)} & \mbox{\bf\small (d)}
\end{array}$
\end{center}
\caption{Solutions of \eqref{eq:main3hom} with summable (Parts (a) and (b)) and non-summable  (Parts (c) and (d)) timestep sequences. Short term and long term dynamics are given in the first and second columns, respectively.}\label{fig:detplots}
\end{figure}

\begin{figure}
\begin{center}
$\begin{array}{@{\hspace{-0.3in}}c@{\hspace{-0.3in}}c}
\mbox{\bf\small Short-term} & \mbox{\bf\small Long-term}\\
\scalebox{0.33}{\includegraphics{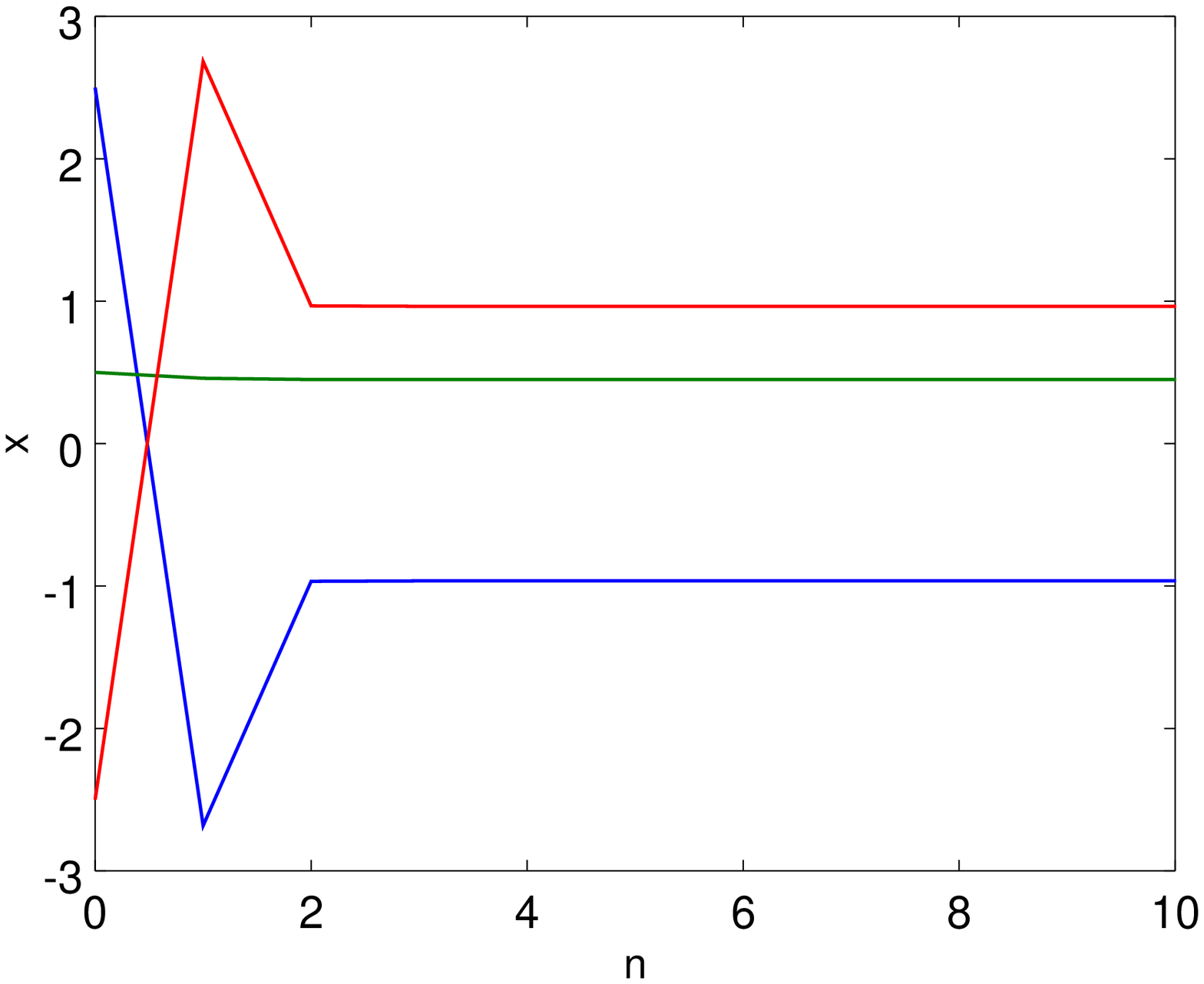}} & \scalebox{0.33}{\includegraphics{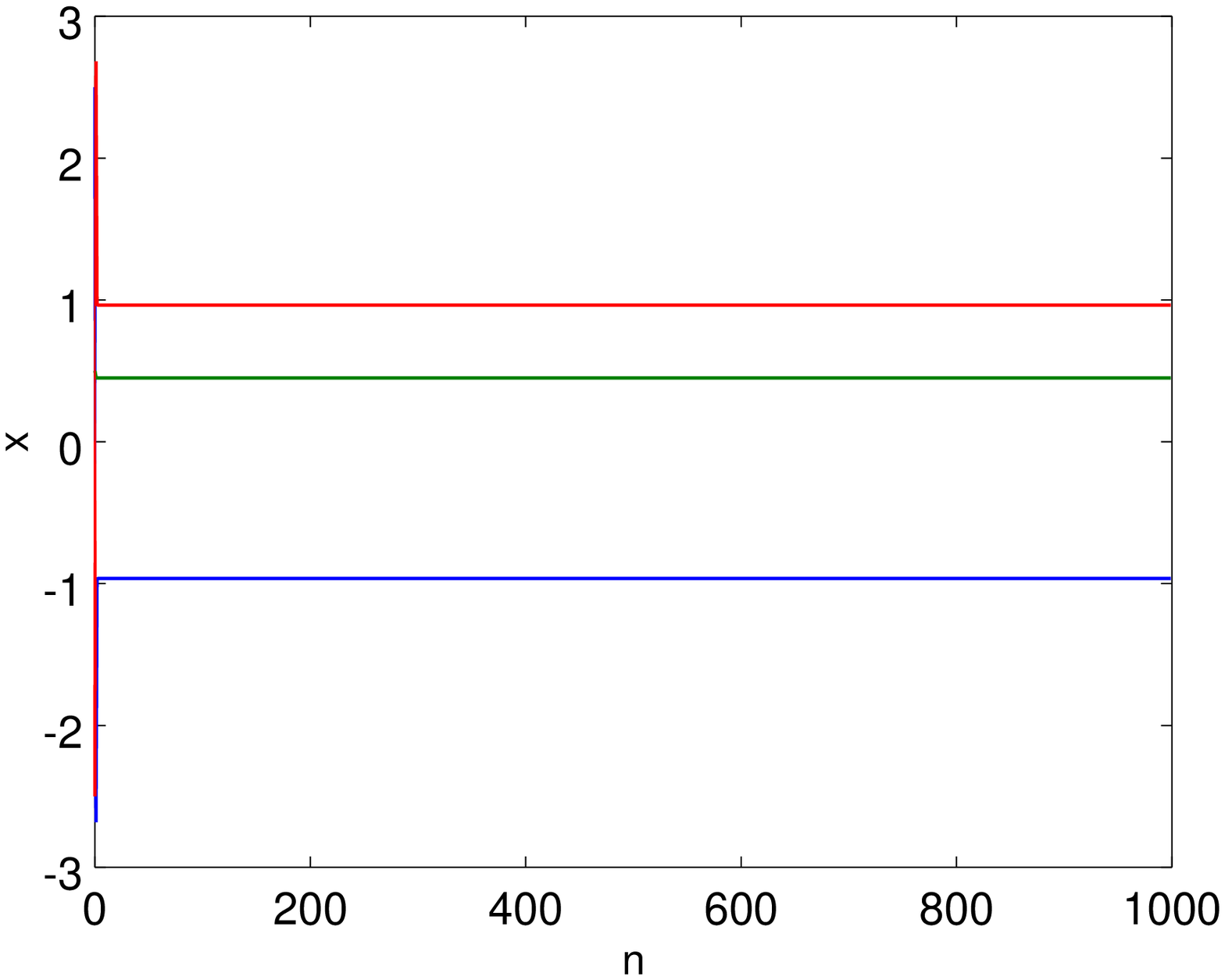}}\\ \mbox{\bf\small (a)} & \mbox{\bf\small (b)}\\
\scalebox{0.33}{\includegraphics{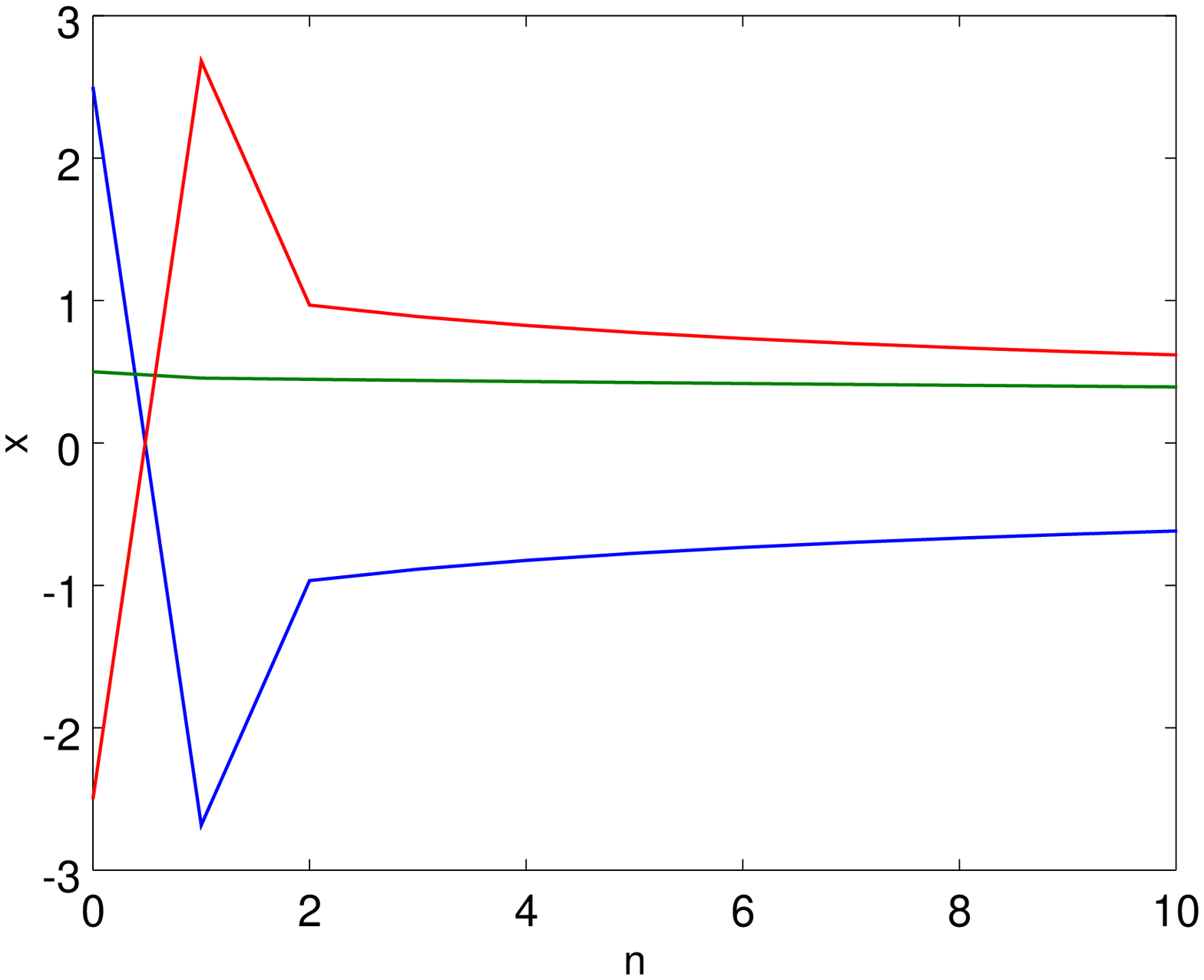}} & \scalebox{0.33}{\includegraphics{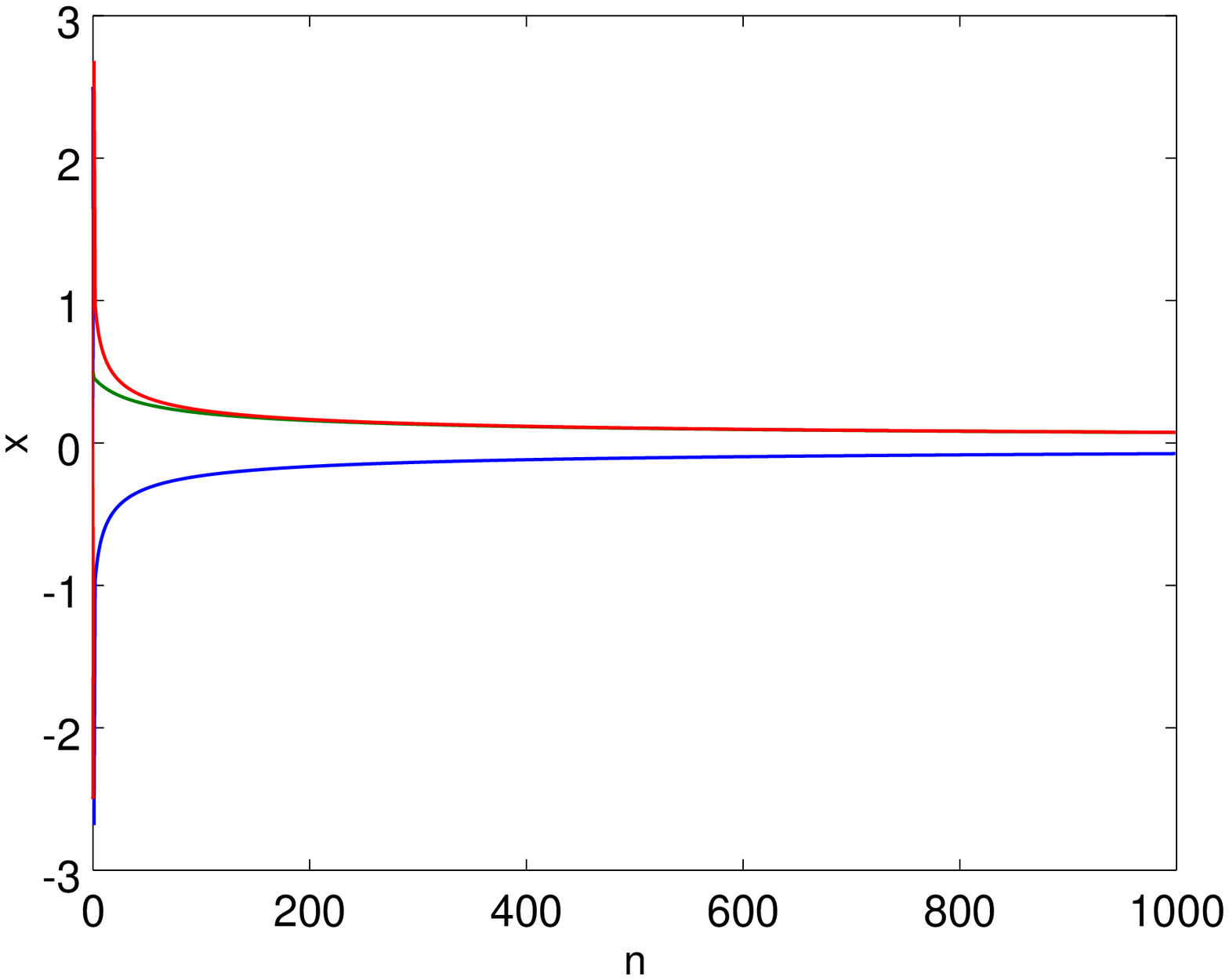}}\\ \mbox{\bf\small (c)} & \mbox{\bf\small (d)}
\end{array}$
\end{center}
\caption{Solutions of \eqref{eq:stoch} with Gaussian perturbation and non-stopped (Parts (a) and (b)) and stopped (Parts (c) and (d)) timestep sequences. Short term and long term dynamics are given in the first and second columns, respectively.}\label{fig:stochplots}
\end{figure}

\begin{figure}
\begin{center}
$\begin{array}{@{\hspace{-0.3in}}c@{\hspace{-0.3in}}c}
\scalebox{0.33}{\includegraphics{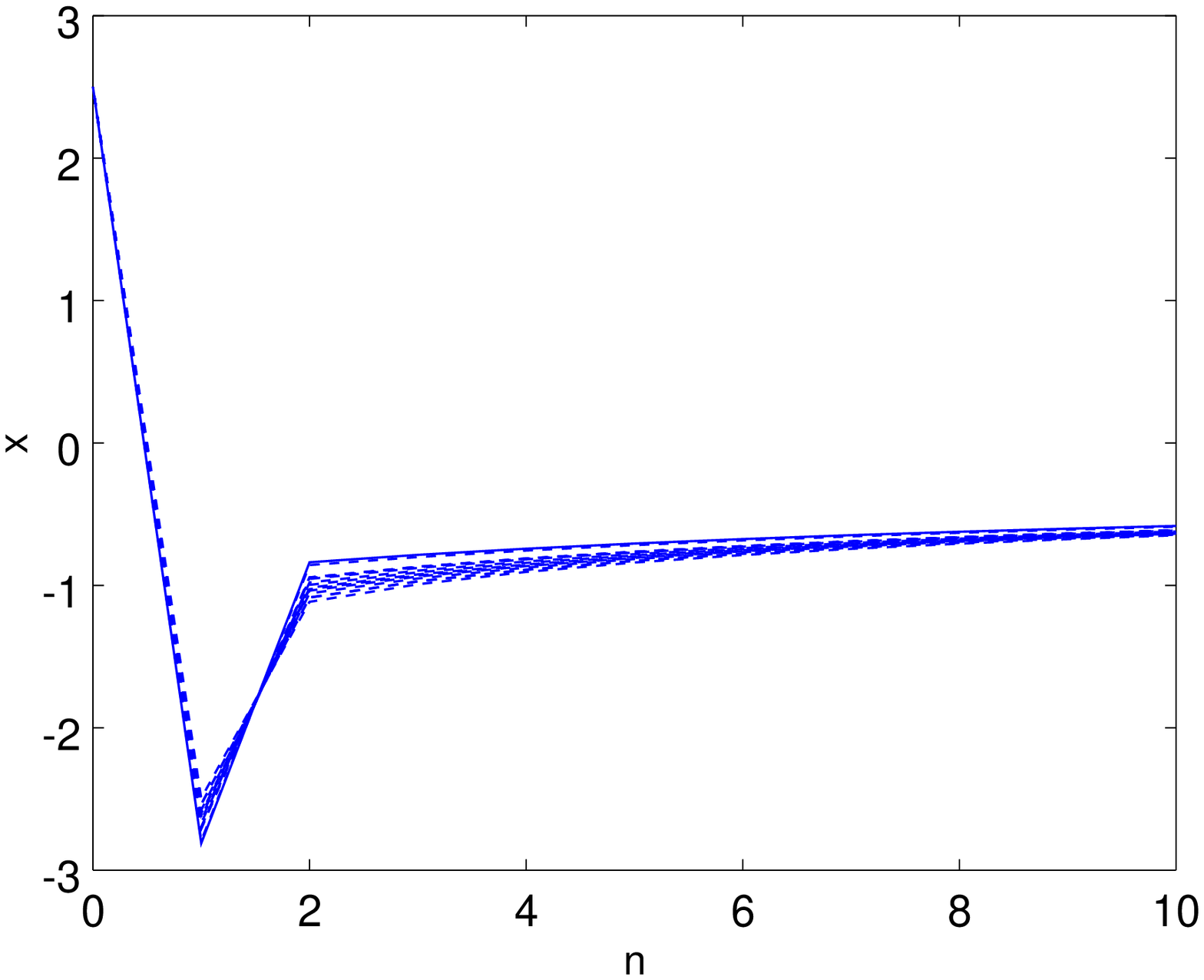}} & \scalebox{0.33}{\includegraphics{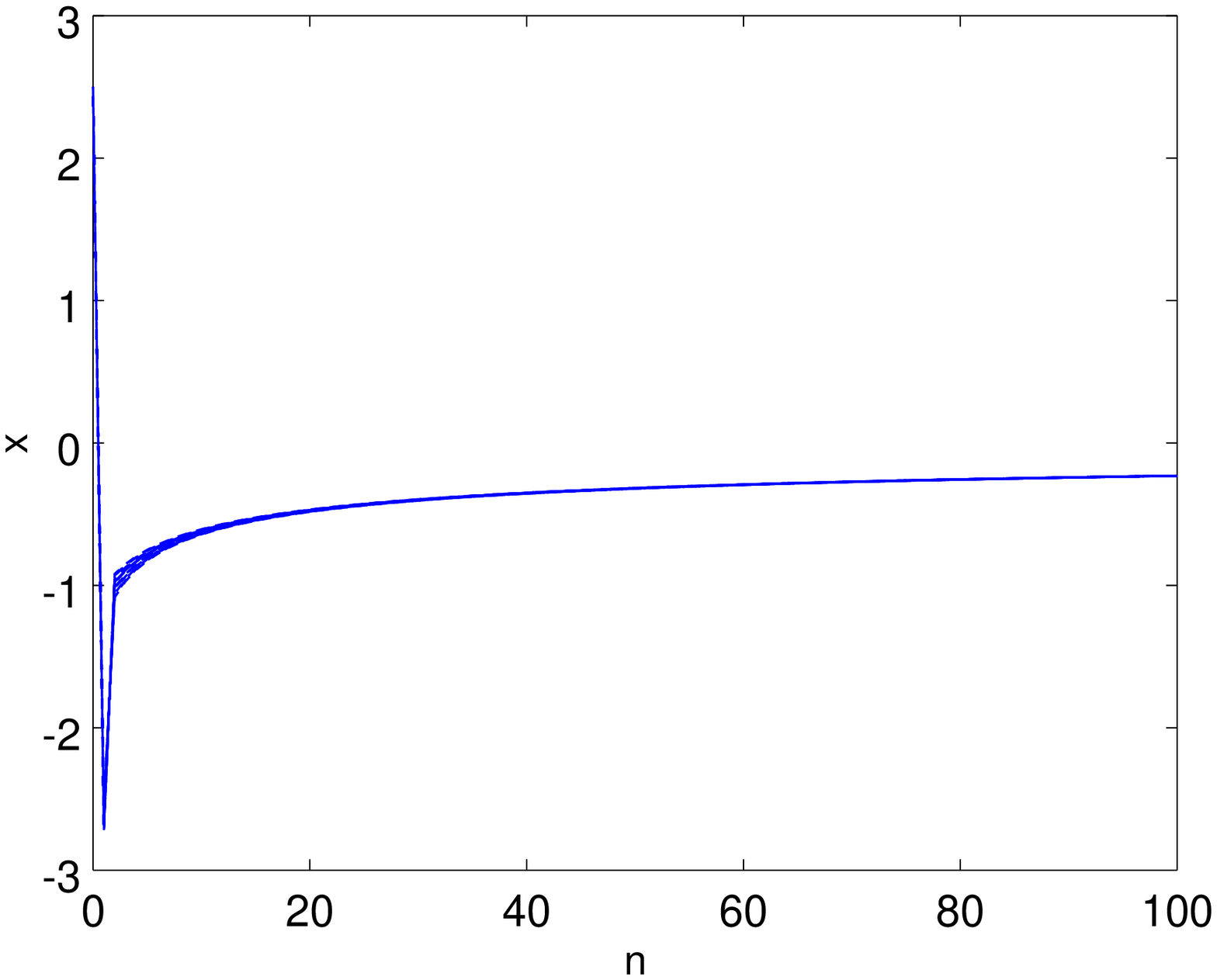}}\\ \mbox{\bf\small (a)} & \mbox{\bf\small (b)}\\
\scalebox{0.33}{\includegraphics{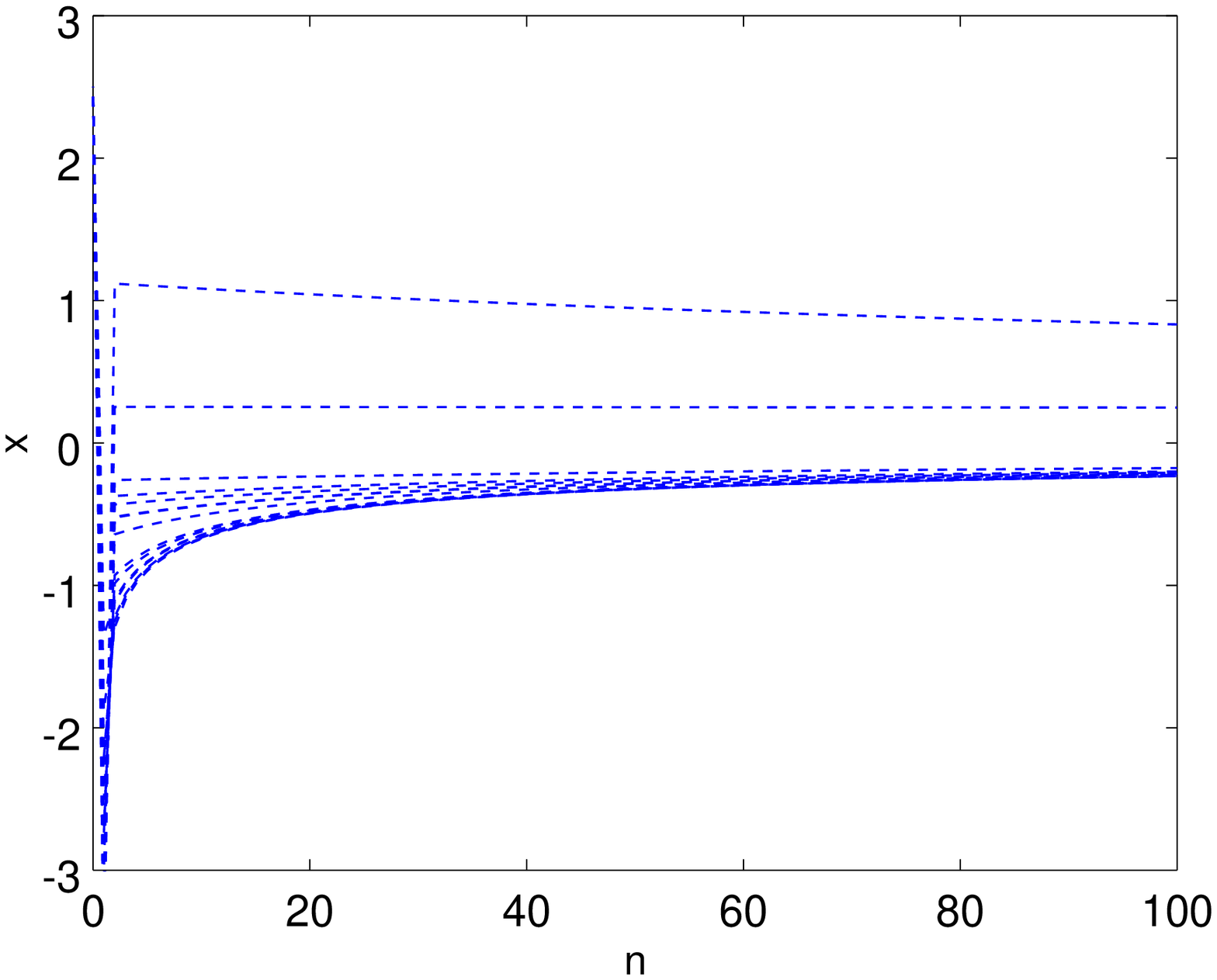}} & \scalebox{0.33}{\includegraphics{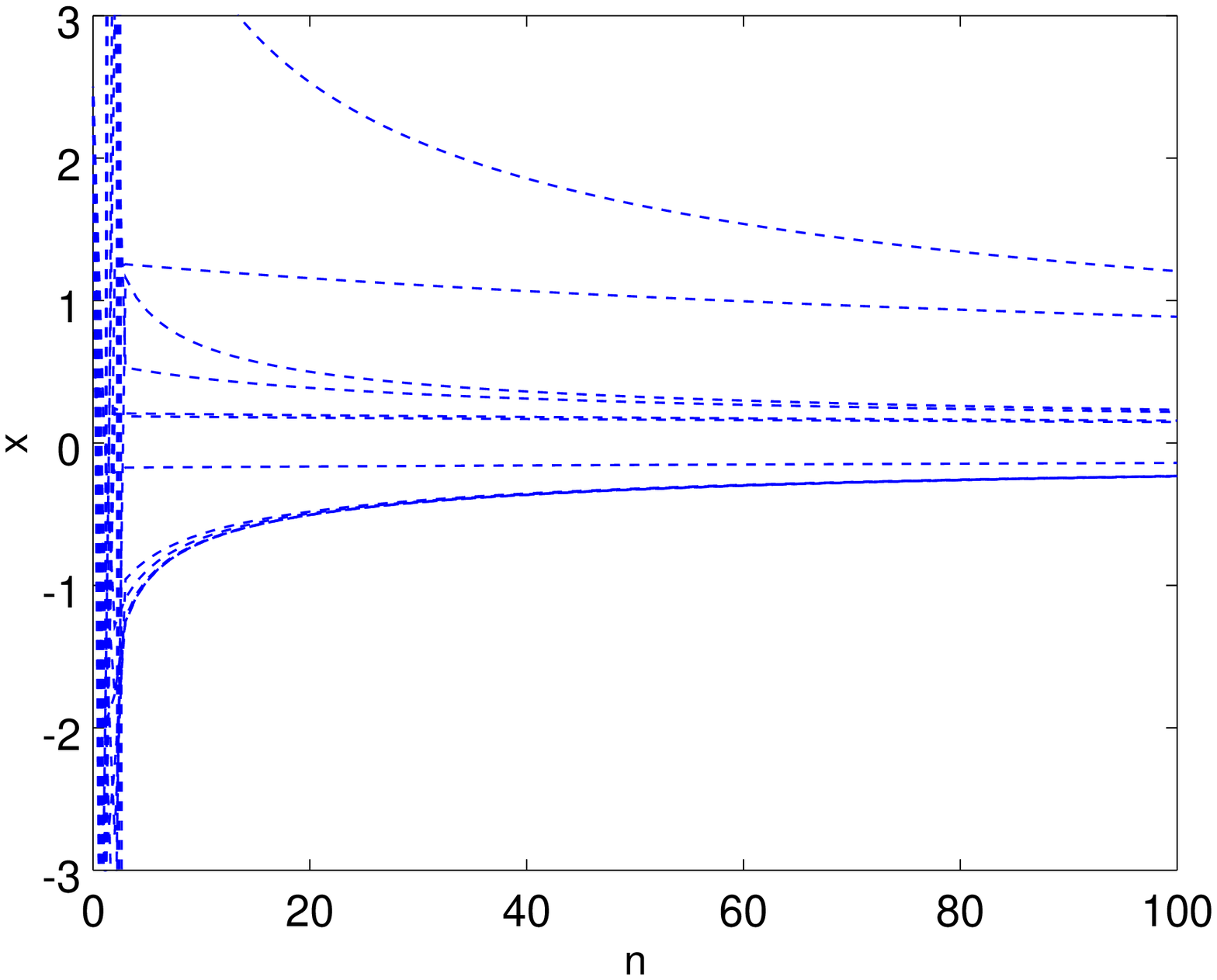}}\\ \mbox{\bf\small (c)} & \mbox{\bf\small (d)}
\end{array}$
\end{center}
\caption{Multiple trajectories of \eqref{eq:stoch} with Gaussian perturbation and stopped timestep sequence. Here, $x_0=2.5$, $\beta=3/2$ (Part (a) short-term and Part (b) long-term behaviour), $\beta=3$ (Part (c), long-term behaviour), and $\beta=5$ (Part (d), long-term behaviour).}\label{fig:stochplotsmanytraj}
\end{figure}

\section{Acknowledgment}
The third author is grateful to the organisers of the 23rd International Conference on Difference Equations and Applications, Timisoara, Romania, who supported her participation.  Discussions at the conference were quite beneficial for this research.

\end{document}